\documentclass[12pt,a4paper]{amsart}

\usepackage{amsmath,amsthm,amssymb,latexsym,a4wide,tikz,multicol,tikz-cd}
\usepackage{arydshln,multirow}
\usepackage{tikz-qtree,tikz-qtree-compat}
\usepackage{mathtools,stmaryrd}
\usepackage{tikz}
\tikzset{font=\small}
\usepackage{enumerate}
\usetikzlibrary{matrix,arrows}
\usetikzlibrary{positioning}
\usetikzlibrary{cd}

\usepackage{mathrsfs}

\newtheorem{theorem}{Theorem} [section]
\newtheorem{lemma}[theorem]{Lemma}
\newtheorem{corollary}[theorem]{Corollary}
\newtheorem{proposition}[theorem]{Proposition}
\theoremstyle{definition}
\newtheorem{definition}[theorem]{Definition}
\newtheorem{remark}[theorem]{Remark}
\newtheorem{example}[theorem]{Example}

\newtheorem{thm}{Theorem}



\usepackage[active]{srcltx}

\usepackage{enumerate}
\numberwithin{equation}{section}

\title[Relating ample categories with Boolean restriction semigroups]{Relating ample and biample topological categories with Boolean restriction and range semigroups}

\author{Ganna Kudryavtseva}
\address{G. Kudryavtseva: University of Ljubljana,
Faculty of Mathematics and Physics, Jadranska ulica~19, SI-1000 Ljubljana, Slovenia / Institute of Mathematics, Physics and Mechanics, Jadranska ulica 19, SI-1000 Ljubljana, Slovenia}
\email{ganna.kudryavtseva\symbol{64}fmf.uni-lj.si}
\thanks{The author was supported by ARIS grants P1-0288 and J1-60025.}

\begin{document}

\begin{abstract}  
We extend the equivalence by Cockett and Garner between restriction monoids and ample categories to the setting of Boolean range semigroups which are non-unital one-object versions of range categories. We show that Boolean range semigroups are equivalent to ample topological categories where the range map $r$ is open, and \'etale Boolean range semigroups are
 equivalent to biample topological categories. These results yield the equivalence between \'etale Boolean range semigroups and Boolean birestriction semigroups and a characterization of when a Boolean restriction semigroup admits a compatible cosupport operation. We also recover the equivalence between Boolean birestriction semigroups and biample topological categories by Kudryavtseva and Lawson. Our technique builds on the usual constructions relating inverse semigroups with ample topological groupoids via germs and slices.
\end{abstract}
\maketitle

\tableofcontents

\section{Introduction}
\subsection{Motivation and purpose} This paper is a contribution to the theory of non-commutative Stone dualities relating \'etale groupoids and categories with non-commutative generalizations of generalized Boolean algebras, such as Boolean inverse or Boolean restriction semigroups \cite{CG21, KL17, L10, L12, LL13, Re07}. The interest to this topic is motivated by the prominent role played by \'etale groupoids in construction of operator algebras \cite{Paterson, R80} and Steinberg algebras \cite{CFST14, St10} with rich structure and extreme properties. Among many others, these include graph and higher rank graph $C^*$-algebras \cite{KP00, KPRR97}, Leavitt path algebras of graphs~\cite{AASM_book} and Kumjian-Pask algebras of higher rank graphs \cite{APCHR13}, which themselves form fast evolving and well established fields of research in analysis and algebra. 

Boolean inverse semigroups were introduced by Lawson \cite{L10,L12}, who  showed that they are in a duality with ample topological groupoids\footnote{Ample topological groupoids are called {\em Boolean groupoids} in \cite{L10, L12}.} which extends the classical Stone duality between Boolean algebras and Stone spaces \cite{Stone37}. Lawson's work develops the ideas of the earlier work by Resende \cite{Re07} where pseudogroups (which are related to Boolean inverse semigroups similarly to as frames are to Boolean algebras) are put in correspondence with localic \'etale groupoids which are non-commutative generalizations of locales. Boolean inverse semigroups, their variations, generalizations and applications were further explored in Lawson and Lenz \cite{LL13}, Kudryavtseva and Lawson \cite{KL17}, Wehrung \cite{W17,W18}, Ara, Bosa, Pardo and Sims \cite{ABPS21}, Lawson and Resende \cite{LR21}, Lawson and Vdovina \cite{LV20}, Cockett and Garner \cite{CG21}, Garner \cite{G23a,G23b}, Steinberg and Szak\'acs \cite{StSz23}, Kudryavtseva \cite{K21}, Cordeiro \cite{C23}, among other works.

It is well known that to construct a Boolean inverse semigroup from an ample topological groupoid, one takes all its compact bislices\footnote{What we call bislices is usually called slices in the literature.}, see \cite[Proposition 4]{L23}. Kudryavtseva and Lawson \cite{KL17} extended this construction from ample groupoids to biample categories, where compact bislices form a Boolean birestriction semigroup\footnote{Boolean birestriction semigroups are called Boolean restriction semigroups in \cite{KL17}.}.
These results were recently applied by de Castro and Machado \cite{CM23} who initiated the study of non self-adjoint operator algebras associated to biample topological categories which generalize groupoid $C^*$-algebras.

Cockett and Garner \cite{CG21} pursued investigating non-regular generalizations of Boolean inverse semigroups much further and established the equivalence between what they term join restriction categories which are hyperconnected over a join restriction category ${\mathcal C}$ with local glueings and partite source-\'etale internal categories in ${\mathcal C}$. As a special case, this includes the equivalence between Boolean restriction monoids and ample topological categories (whose space of identities is compact). Unlike the earlier dualities for Boolean birestriction and inverse semigroups, this equivalence may look unbalanced in that restriction monoids therein are one-sided objects (they have the unary support operation $^*$, while no unary cosupport operation $^+$), but are put in correspondence with categories which are two-sided objects (they have both the domain map $d$ and the range map $r$). This may look even more perplexing in the light of the result of Gould and Hollings \cite{GH09}, which is an extension of the Ehresmann-Schein-Nambooripad theorem (see, for example, \cite{Lawson_book}) and states that restriction semigroups are equivalent to  inductive {\em constellations} which are one-sided variants of categories without the range map (see Gould and Stokes \cite{GS17,GS22,GS24}). 

Since the Cockett-Garner equivalence \cite{CG21} puts one-sided objects (restriction monoids) into correspondence with two-sided objects (categories), it is reasonable to wonder about possible enrichment of this equivalence to the setting where restriction monoids admit a compatible cosupport operation. It is natural to expect that this operation and its properties be encoded on the topological side by suitable properties of the range map $r$ of the corresponding ample category. However, these kinds of questions were not addressed in~\cite{CG21} and it is one of the aims of this work to consider them. 

Another purpose of this work is to extend the Cockett-Garner equivalence   to the non-unital setting of restriction {\em semigroups}. We show that appropriate generalizations of Boolean restriction monoids turn out to be {\em Boolean restriction semigroups with local units}. We also give special attention to the correspondence between various kinds of morphisms between algebraic and topological objects. The morphisms we work with at the topological side are {\em cofunctors} between ample categories and extend cofunctors between ample groupoids (see~\cite{BEM12} and also~\cite{CG21, K12}).

Finally, we aim this work to be the most tailored to a reader working with groupoid $C^*$-algebras and Steinberg algebras, and who is familiar with the constructions of an inverse semigroup from an ample groupoid and vice versa using compact bislices and germs (see Exel~\cite{Exel08} and Steinberg~\cite{St10}). That is why we work with ample topological categories and their subclasses, and do not consider a wider generality.  We anticipate, however, that our results admit extension to the level of generality of \cite{CG21} with restriction categories replaced by suitable {\em range semicategories}. 

\subsection{Structure and main results} In Section \ref{s:stone_gba} we recall the classical Stone duality between generalized Boolean algebras and locally compact Stone spaces. Then in Section \ref{s1:preliminaries} we provide the necessary background on Ehresmann, biEhresmann, restriction and birestriction semigroups and, motivated by range categories of Cockett, Guo and Hofstra \cite{CGH12, CGH12a},  introduce {\em range semigroups} which are the non-unital versions of one-object range categories. Range semigroups are simply biEhresmann  semigroups $(S,\cdot, ^*, ^+)$ such that $(S,\cdot, ^*)$ is a restriction semigroup. The prototypical example of a (Boolean) range semigroup is the semigroup ${\mathcal{PT}}(X)$ of all partial self-maps of a set $X$. In Section \ref{s:determ} we study the com\-pa\-ti\-bi\-li\-ty and the bicompatibility relations, joins, and also bideterministic elements and partial isomorphisms of \cite{CGH12, G23b} with the emphasis on the comparison between these two classes of elements in a range semigroup.  Further in Section \ref{s:brs} we define our central objects, {\em Boolean} and {\em preBoolean restriction semigroups with local units}, and several types of morphisms between them. PreBoolean semigroups are more general than Boolean ones in that the  requirement of existing of all binary compatible joins is relaxed to that of existing of joins of all pairs of elements with a common upper bound. These semigroups include not only Boolean restriction semigroups, but also Boolean birestriction semigroups which are closed under finite bicompatible joins rather than with respect to finite compatible joins. This enables us in Section \ref{s:dualities} to derive the equivalences for both Boolean restriction and Boolean birestriction semigroups from  suitable adjunctions involving preBoolean semigroups.
 
Section \ref{s:cat_to_sem} concerns definitions and first properties of \'etale and bi\'etale topological categories. 
In Section \ref{s:ample} we define ample topological categories and show that the set $C^a$ of compact slices, together with the operations of the multiplication of slices and $^*$ given by  $A^* = ud(A)$, forms a Boolean restriction semigroup with local units (see Proposition \ref{prop:Boolean}). We also look at the special cases where $r$ is open or a local homeomorphism
 and show that the  objects that arise at the algebraic side  are {\em Boolean range semigroups} and {\em \'etale Boolean range semigroups}, respectively (see Proposition \ref{prop:range} and Proposition \ref{prop:etale}). We also define several types of morphisms between ample categories and describe their algebraic counterparts. 

In Section \ref{s:sem_to_cat} we reverse the constructions of Section \ref{s:ample} and  assign to a preBoolean restriction semigroup with local units its {\em category of germs} which is inspired by Exel's tight groupoid of an inverse semigroup with zero \cite{Exel08} (which is equivalent to the groupoid of tight filters, see~\cite{BCHJL22}). In fact, a similar construction can be carried out for an arbitrary  restriction semigroup with zero which is a projection, but we postpone pursuing this theme until future work. We prove that the category of germs is an ample category (see Corollary~\ref{cor:ample_cat}) and in the special cases where $S$ is a range semigroup or an \'etale range semigroup, the map $r$ of  its category of germs is open or a local homeomorphism, respectively (see Proposition~\ref{prop:open1} and Proposition \ref{prop:lh1}). 

Finally, in Section \ref{s:dualities} we formulate and prove the adjunction and the equivalence theorems, which are the main results of this paper (see Theorems \ref{th:adjunction1},  \ref{th:eq1}, \ref{th:eq2} and \ref{th:eq3}) and are briefly outlined below. We call an ample topological category {\em strongly ample} if its range map $r$ is open and {\em biample} if $r$ is a local homeomorphism.

\begin{thm}\label{thm1} \mbox{}
\begin{enumerate}
\item There is an adjunction between the category of preBoolean restriction semigroups with local units and the category of ample categories.
\item The adjunction of part (1) restricts to the adjunction between the category of preBoolean range semigroups and the category of strongly ample categories.
\item The adjunction of part (2) restricts to the adjunction between the category of \'etale preBoolean range semigroups and the category of biample categories.
\end{enumerate}    
\end{thm}

These adjunctions give rise to the following equivalences.

\begin{thm}\label{thm2} The following pairs of categories are equivalent:
\begin{enumerate}
\item the category of Boolean restriction semigroups with local units and the category of ample categories;
\item the category of Boolean range semigroups and the category of strongly ample categories;
\item  the category of \'etale Boolean range semigroups and the category biample categories;
\item the category of Boolean birestriction semigroups and the category of biample categories.
\end{enumerate}    
\end{thm}

As Proposition \ref{prop:6j3} shows, our \'etale Boolean range semigroups (which are defined as those Boolean range semigroups each of whose elements is a finite join of bideterministic elements) are a natural generalization of \'etale Boolean restriction monoids considered by Garner in \cite{G23b} (and our results imply that Garner's \'etale Boolean restriction monoids admit a compatible cosupport operation). By Proposition \ref{prop:det} a biample category $C$  gives rise to the \'etale Boolean range semigroup $C^a$ of all compact slices and to the birestriction semigroup $\widetilde{C}^a$ of all compact bislices which yields the equivalence between \'etale Boolean range semigroups and Boolean birestriction semigroups (see Theorem \ref{th:eq3}). It follows that an \'etale Boolean range semigroup can be reconstructed from its birestriction semigroup of bideterministic elements. In particular, a groupoidal \'etale Boolean range semigroup can be reconstructed from its inverse semigroup of partial isomorphisms, cf. \cite{G23b, L24}.

Part (4) of Theorem \ref{thm2} recovers the non-commutative Stone duality for Boolean birestriction semigroups given in \cite[Theorem 8.19]{KL17}. Note that the (unital version of the) latter duality {\em can not} be recovered from the results of \cite{CG21}, because while treating (generalizations of) restriction and inverse monoids, \cite{CG21} does not treat the  biunary setting at the algebraic side of the equivalences.

\section{Stone duality for generalized Boolean algebras} \label{s:stone_gba}
Here we briefly recall the duality between generalized Boolean algebras and locally compact Stone spaces \cite{D64,Stone37} (see also \cite{L23}) which we generalize in this paper to several non-commutative settings.
Recall that a {\em generalized Boolean algebra} is a relatively complemented distributive lattice with the bottom element $0$. By ${\mathbb B} = \{0,1\}$ we denote the smallest non-zero generalized Boolean algebra, with the operations $\vee$, $\wedge$ and $\setminus$ of join, meet and relative complement given in the following tables.
\vspace{2.5mm}

\begin{minipage}[c]{0.32\textwidth}
\centering
\begin{tabular}{|c||c|c|}
\hline
 $\vee$ & 0 & 1 \\
\hline  \hline
 0 & 0 & 1  \\
\hline 
 1 & 1 & 1 \\
 \hline
\end{tabular}
\end{minipage}
\begin{minipage}[c]{0.32\textwidth}
\centering
\begin{tabular}{|c||c|c|}
\hline
 $\wedge$ & 0 & 1 \\
\hline  \hline
 0 & 0 & 0  \\
\hline 
 1 & 0 & 1 \\
 \hline
\end{tabular}
\end{minipage}
\begin{minipage}[c]{0.32\textwidth}
\centering
\begin{tabular}{|c||c|c|}
\hline
 $\setminus$ & 0 & 1 \\
\hline  \hline
 0 & 0 & 0  \\
\hline 
 1 & 1 & 0 \\
 \hline
\end{tabular}
\end{minipage}

\vspace{2.5mm}

A Hausdorff space $X$ is a {\em locally compact
Stone space} if it has a basis of compact-open sets. A morphism $E\to F$ between generalized Boolean algebras is said to be {\em proper} if for every $f\in F$ there is $e\in E$ such that $\varphi(f)\geq e$. 
A continuous map  between topological spaces is called {\em proper} if the inverse images of compact sets are compact sets.

 We denote the category of generalized Boolean algebras and proper morphisms between them by ${\mathsf{GBA}}$, and by ${\mathsf{LCBS}}$ the category of locally compact Stone spaces and proper continuous maps between them.

Let $E$ be a generalized Boolean algebra. A {\em prime character} of $E$ is a non-zero morphism of generalized Boolean algebras $f\colon E\to {\mathbb B}$. 
By $\widehat{E}$ we denote the set of all prime characters of $E$.
Prime characters are in a bijection with {\em prime filters} of $E$ via the assignment 
$\varphi\mapsto \varphi^{-1}(1)$. The set $\widehat{E}$ is a locally compact Stone space with the basis of the topology given by the sets $D_e = \{f\in \widehat{E}\colon f(e)=1\}$, where $e$ runs through $E$. The assignment $E\mapsto \widehat{E}$ gives rise to a contravariant functor ${\mathsf{F}}\colon {\mathsf{GBA}} \to {\mathsf{LCBS}}$ which acts on morphism by taking $f\colon E\to F$ to $f^{-1}\colon \widehat{F}\to \widehat{E}$, where $f^{-1}(\varphi) = \varphi\circ f$.

Let $X$ be a locally compact Stone space. By $\widehat{X}$ we denote the generalized Boolean algebra of all compact-open sets of $X$. The assignment $X\mapsto \widehat{X}$ gives rise to a contravariant functor ${\mathsf G}\colon {\mathsf{LCBS}} \to {\mathsf{GBA}}$ which takes $f\colon X\to Y$ to $f^{-1}\colon \widehat{Y}\to \widehat{X}$.

\begin{theorem} \cite{D64, Stone37} (Stone duality for generalized Boolean algebras) \label{th:Stone_classical}
The functors ${\mathsf{F}}$
and ${\mathsf G}$ establish a dual equivalence between the categories ${\mathsf{GBA}}$ and ${\mathsf{LCBS}}$. The natural isomorphisms $\eta\colon 1_{{\mathsf{GBA}}} \to {\mathsf G}{\mathsf F}$ and $\varepsilon \colon {\mathsf F}{\mathsf G} \to  1_{{\mathsf{LCBS}}} $ are given by:
\begin{itemize}
\item $\eta_E(a) = \{\varphi\in \widehat{E}\colon \varphi(a)=1\}$, where $E$ is an object of ${\mathsf{GBA}}$ and $a\in E$.
\item $\varepsilon_X(\varphi_x) = x$, where $\varphi_x\in {\mathsf F}{\mathsf G}(X)$ is defined by:
$$
\text{ for all } A\in \widehat{X}: \varphi_x(A) = 1 \Leftrightarrow x\in A,
$$ where $X$ is an object of ${\mathsf{LCBS}}$ and $x\in X$.
\end{itemize}
\end{theorem}

Since the proofs of our main results in Section~\ref{s:dualities} rely on~Theorem \ref{th:Stone_classical}, for the sake of completeness, we briefly sketch the arguments why $\eta_E$ and $\varepsilon_X$  are isomorphisms. The definition of the sets $D_e$ yields that $D_{e\vee f} = D_e\cup D_f$ and this extends to any finite non-empty joins. If $A$ is a compact-open set of $\widehat{E}$, it is a finite union of sets of the form $D_e$, so that $A = \cup_{i=1}^nD_{e_i} = D_{\vee_{i=1}^n}e_i$, and thus $\eta_E$ is surjective. If $e\neq f$ are elements of $E$, there is an ultrafilter of $E$ which contains $a$ and avoids $b$  (see \cite[Corollary~2]{L23}). But ultrafilters of generalized Boolean algebras coincide with prime filters (see \cite[Proposition~1.6]{LL13}), so that $D_e\neq D_f$. It follows that $\eta_E$ is a bijection. To show that $\varepsilon_X$ is well defined and bijective, we first note that it is routine to check that $\varphi_x$ is indeed a prime character of $\widehat{X}$. To show that the map $x\mapsto \varphi_x$ is a bijection (and then $\varepsilon_X$ is its inverse bjection), suppose first that $x\neq y$. Because $X$ is Hausdorff, there is a neighbourhood $U$ of $x$ which avoids $y$. Since compact-open sets form a basis of the topology on $X$, there is $A\in \widehat{X}$ such that $A\subseteq U$ and $x\in A$. It follows that $\varphi_x(A)=1$ and $\varphi_y(A)=0$, so that $\varphi_x\neq \varphi_y$. Let $\psi \in {\mathsf F}{\mathsf G}(X)$. Then $\psi^{-1}(1)$ is a prime filter of $\widehat{X}$. Since $X$ is compact, the finite intersection property implies that $\psi^{-1}(1)$ has a non-empty intersection. So there is $x\in X$ such that $\psi^{-1}(1) \supseteq \varphi_x^{-1}(1)$. Since prime filters coincide with ultrafilters (see \cite[Proposition~1.6]{LL13}), maximality of $\varphi_x^{-1}(1)$ implies that $\psi^{-1}(1) = \varphi_x^{-1}(1)$, so that $\psi = \varphi_x$. For a more detailed account we refer the reader to \cite[Proposition~3, Theorem~6]{L23}.

\section{Restriction, birestriction and range semigroups}\label{s1:preliminaries}
\subsection{Ehresmann, biEhresmann, restriction and birestriction semigroups} 
Ehresmann, biEhresmann, restriction and birestriction semigroups are the main objects of study of this paper and appear on the algebraic side of our dualities. They are a subject of active investigations, see  \cite{BGG15,DKK21,EG21,Kam11,K19,KL23,S17,S18} and references therein, and for a survey we refer the reader to  \cite{G10}.

In this paper, algebras are meant in the sense of universal algebra, so an {\em algebra} $A$ is an ordered pair $(A; F)$ where $A$ is a non-empty set and $F$ a collection of finitary operations on $A$, see \cite[Definition 1.1]{McKMcNT87}. If $(A; f_1, \dots, f_n)$ is an algebra, by its {\em signature} we mean the ordered $n$-tuple $(f_1, \dots, f_n)$ and by its {\em type} the ordered $n$-tuple $(\rho(f_1),\dots, \rho(f_n))$ where $\rho(f_i)$ is the arity of the operation $f_i$ for each $i\in \{1,\cdots, n\}$. If $(A; f_1, \dots, f_n)$ is an algebra of type $(\rho(f_1),\cdots, \rho(f_n))$, we say that it is a $(\rho(f_1),\dots, \rho(f_n))$-algebra. Furthermore, if this does not cause ambiguity, we often denote $(A; f_1, \dots, f_n)$ simply by $A$.

\begin{definition} [Ehresmann, coEhresmann and biEhresmann semigroups] An {\em Ehresmann semigroup}\footnote{Ehresmann and coEhresmann semigroups appear in the literature as {\em right Ehresmann semigroups} and {\em left Ehresmann semigroups}.} is an algebra $(S; \cdot \,, ^*)$, where $(S;\cdot)$ is a semigroup, $^*$ is a unary operation on $S$ such that the following identities hold:
\begin{equation}\label{eq:axioms_star}
xx^*=x, \qquad x^*y^*=y^*x^* = (x^*y^*)^*, \qquad (xy)^*=(x^*y)^*.
\end{equation}

Dually, a {\em coEhresmann semigroup} is an algebra $(S; \cdot \,, ^+)$, where $(S;\cdot)$ is a semigroup, $^+$ is a unary operation on $S$ such that the following identities hold:
\begin{equation}\label{eq:axioms_plus}
x^+x=x, \qquad x^+y^+=y^+x^+=(x^+y^+)^+, \qquad (xy)^+=(xy^+)^+.
\end{equation}

A {\em biEhresmann semigroup}\footnote{BiEhresmann semigroups appear in the literature as {\em two-sided Ehresmann semigroups} or as {\em Ehresmann semigroups}.} is an algebra $(S; \cdot\, , ^*, ^+)$, where $(S;\cdot\, , ^*)$ is an Ehresmann semigroup, $(S;\cdot\, , ^+)$ is a  coEhresmann semigroup and the operations $^*$ and $^+$ are connected by the following identities:
\begin{equation}\label{eq:axioms_common}
(x^+)^*=x^+,\qquad (x^*)^+=x^*.
\end{equation}
\end{definition}

The identities $x^* = x^*x^*$ and $(x^*)^* = x^*$ follow from \eqref{eq:axioms_star}, and their dual identities $x^+ = x^+x^+$ and $(x^+)^+ = x^+$  follow from  \eqref{eq:axioms_plus}. The operation $^*$ will be sometimes referred to as the {\em support operation}, and the operation $^+$ as the {\em cosupport operation}.

\begin{definition}[Restriction, corestriction and birestriction semigroups] 
A {\em restriction semigroup}\footnote{Restriction and corestriction semigroups appear in the literature  as {\em right  restriction semigroups} (or {\em weakly right ample semigroups}) and {\em left restriction semigroups} (or {\em weakly left ample semigroups}), respectively.} is an algebra $(S; \cdot \,, ^*)$ which is an Ehresmann semigroup and in addition satisfies the identity
\begin{equation}\label{eq:axioms_star_restr}
x^*y = y(xy)^*.
\end{equation}
Dually, a {\em corestriction semigroup} is an algebra $(S; \cdot \,, ^+)$ which is a coEhresmann semigroup and satisfies the identity
\begin{equation}\label{eq:axioms_plus_restr}
xy^+ = (xy)^+x.
\end{equation}
A {\em birestriction semigroup}\footnote{Just as with biEhresmann semigroups, birestriction semigroups appear in the literature  as {\em two-sided restriction semigroups} or as  {\em restriction semigroups}.} is an algebra $(S; \cdot\, , ^*, ^+)$, 
where $(S;\cdot\, , ^*)$ is a restriction semigroup, $(S;\cdot\, , ^+)$ is a  corestriction semigroup and the operations $^*$ and $^+$ are connected by the  identities \eqref{eq:axioms_common}.
\end{definition}

If these algebras possess an identity element, the word `semigroup' in their name is replaced by the word `monoid', e.g., a restriction monoid is a restriction semigroup which has an identity element. We denote the identity element of a monoid by $1$. Note that in an Ehresmann monoid we have $1^*=1$. Dually, $1^+=1$ holds in a coEhresmann monoid.

\begin{remark}
In the category theory literature multi-object generalizations of Ehresmann, coEhresmann, biEhresmann, restriction, corestriction and birestriction monoids are known as {\em support}, {\em cosupport}, {\em bisupport}, {\em restriction}, {\em corestriction} and {\em birestriction categories}, respectively  (see, e.g., \cite{CG21,  CGH12, CL02, CL24, HL21}). Furthermore, the support and the cosupport of an element $s$ are often denoted by $\bar{s}$ and $\hat{s}$, respectively.  \end{remark}

Recall that a semigroup $S$ is called an {\em inverse semigroup} \cite{Lawson_book} if for each $a\in S$ there exists a unique
$b\in S$ such that $aba=a$ and $bab=b$. The element $b$ is called the {\em inverse} of $a$ and is denoted by $a^{-1}$. 
If $(S; \cdot)$ is an inverse semigroup and $a\in S$, putting $a^+=aa^{-1}$ and $a^*=a^{-1}a$ makes $(S; \cdot, ^+, ^*)$ a birestriction semigroup (and thus also a biEhresmann semigroup). Multi-object generalizations of inverse monoids are known as {\em inverse categories}, see \cite{CL02, CH23, DWP18}.
	
Ehresmann and coEhresmann semigroups are usually considered as $(2,1)$-algebras, and biEhresmann semigroups as $(2,1,1)$-algebras. 
Morphisms and subalgebras of all such algebras are taken with respect to their signatures. To emphasize this, we sometimes refer to morphisms of, e.g., biEhresmann semigroups as $(2,1,1)$-morphisms.
		
Let $S$ be an Ehresmann semigroup. The set 
$$
P(S) = \{s^*\colon s\in S\}
$$
is closed with respect to the multiplication and is a semilattice. It is called 
the {\em projection semilattice} of $S$ and its elements are called {\em projections}. 
Dually, the projection semilattice of a  coEhresmann semigroup $S$ is
defined as 
$$
P(S)=\{s^+\colon s\in S\}.
$$
If $S$ is biEhresmann semigroup, we have $P(S)=\{s^+\colon s\in S\} = \{s^*\colon s\in S\}$.

It follows from \eqref{eq:axioms_star} that the semilattice $P(S)$ of an Ehresmann semigroup $S$ can be equivalently defined as the set of all $s\in S$ satisfying $s^*=s$, and dually for a coEhresmann semigroup. Furthermore, if $s\in P(S)$ then $s$ is an idempotent. Analogues of projections in restriction and support categories are called {\em restriction idempotents}, see \cite[page 418]{CGH12}.

We point out the following useful identities which easily follow from the definitions.

\begin{enumerate}
\item[1.] Let $S$ be an Ehresmann semigroup. Then
\begin{equation}\label{eq:rule1r}
(se)^* = s^*e \text{ for all }  s\in S \text{ and } e\in P(S).
\end{equation}
\item[2.] Let $S$ be a coEhresmann semigroup. Then
\begin{equation}\label{eq:rule1l}
(es)^+ = es^+ \text{ for all }  s\in S \text{ and } e\in P(S).
\end{equation}
\item[3.] Let $S$ be a  restriction semigroup. Then
\begin{equation}\label{eq:ample_r}
es = s(es)^* \text{ for all } s\in S \text{ and } e\in P(S).  
\end{equation}
\item[4.] Let $S$ be a  corestriction semigroup. Then
\begin{equation}\label{eq:ample_l}
se = (se)^+s \text{ for all } s\in S \text{ and } e\in P(S).
\end{equation}
\end{enumerate}

\begin{definition} [Natural partial orders] \mbox{}
 For elements $s,t$ of an Ehresmann semigroup $S$ we say $s\leq t$ if there is $e\in P(S)$ such that $s=te$ or, equivalently, if $s = ts^*$.
 Dually, for elements $s,t$ of a coEhresmann semigroup $S$ we say $s\leq' t$ if there is $e\in P(S)$ such that $s=et$ or, equivalently, if $s = s^+t$. The relations $\leq$ and $\leq'$ are partial orders called the {\em natural partial orders.} 
\end{definition}

\begin{remark}
If $S$ is an Ehresmann semigroup, one can define $s\leq' t$ if there is $e\in P(S)$ such that $s=et$. However, in the absence of the operation $^+$, this relation does not need to be reflexive or antisymmetric. That is why we consider the relation $\leq'$ only in coEhresmann semigroups, where it is a partial order.
\end{remark}

A biEhresmann semigroup $S$ is equipped with both of the natural partial orders and their restriction to $P(S)$ is the underlying partial order of the semilattice $P(S)$ given by $e\leq f$ if and only if $e=ef=fe$. If $S$ is a birestriction semigroup, the partial orders $\leq$ and $\leq'$ coincide.
The following fact will be used throughout, possibly without further mention.

 \begin{lemma} \label{lem:23} Let $S$ be an Ehresmann semigroup and $a,b\in S$. Then $a\leq b$ implies $a^*\leq b^*$. A dual statement holds in  coEhresmann semigroups.
 \end{lemma}

\begin{proof}
If $a=ba^*$ then applying \eqref{eq:axioms_star} we have $a^*=(ba^*)^*=(b^*a^*)^*=b^*a^*\leq b^*$.   
\end{proof}

\begin{lemma}\label{lem:1} Let $S$ be a restriction semigroup. 
\begin{enumerate}
\item If $e,f\in P(S)$, $s\in S$ and $e\leq f$ then $(es)^* \leq (fs)^*$.
\item  If $s\leq t$ and $u\in S$ then $su \leq tu$ and $us \leq ut$.
\end{enumerate}
Dual statements hold for corestriction semigroups and the partial order $\leq'$.
\end{lemma}

\begin{proof} (1) Using \eqref{eq:rule1r} and \eqref{eq:ample_r}, we have $(es)^*(fs)^* = (es(fs)^*)^* = (efs)^* = (es)^*$.

(2) Since $s = ts^*$, we have $su = ts^*u = tu(s^*u)^* \leq tu$ and $us = uts^* \leq ut$.
\end{proof}

\subsection{Range semigroups}
Suppose that $(S, \cdot)$ is a semigroup and $^*$, $^+$ are unary operations on $S$ such that $(S, \cdot, ^*)$ is an Ehresmann semigroup
and $(S, \cdot, ^+)$ is coEhresmann semigroup. We say that the operations $^*$ and $^+$ are {\em compatible} provided that 
\eqref{eq:axioms_common} holds, that is, if 
$(S, \cdot, ^*, ^+)$ is a biEhresmann semigroup. We say that an Ehresmann semigroup $(S, \cdot, ^*)$ {\em admits a compatible coEhresmann structure $(S, \cdot, ^+)$} if there is a unary operation $^+$ on $S$ such that  $(S, \cdot, ^*, ^+)$ is a biEhresmann semigroup.

Inspired by the notion of a range category of \cite{CGH12}, we make the following definition, which is central for our paper.

\begin{definition} (Range semigroups)
By a {\em range semigroup} we mean a biEhresmann semigroup $(S, \cdot, ^*, ^+)$ such that $(S, \cdot, ^*)$ is a restriction semigroup.
\end{definition}

In other words, range semigroups are precisely those restriction semigroups which admit a compatible coEhresmann structure (with a compatible cosupport operation added to the signature). The multi-object generalization of range semigroups should be naturally termed {\em range semicategories}. The next statement follows from \cite[Subsection 2.11]{CGH12}, but we provide it with a proof, for completeness. 

\begin{proposition}
Let $(S,\cdot,^*)$ be an Ehresmann semigroup. If it admits a compatible  coEhresmann structure then the unary operation $^+$, which makes $(S,\cdot,^*,^+)$ a biEhresmann semigroup, is unique.
\end{proposition}

\begin{proof}
 Suppose that $^+$ and $^{\oplus}$ are unary operations on $S$ such that $(S,\cdot,^*,^+)$ and    $(S,\cdot,^*,^{\oplus})$ are biEhresmann semigroups. Let $s\in S$. Using \eqref{eq:axioms_plus} and \eqref{eq:rule1l}, we have $s^+ = (s^{\oplus}s)^+ = s^{\oplus}s^+$ and similarly $s^{\oplus} = s^+s^{\oplus}$.
Observe that $P(S)=\{s^*\colon s\in S\} = \{s^{+}\colon s\in S\} = \{s^{\oplus}\colon s\in S\}$. Since projections commute, we have that $s^{\oplus}s^+ = s^+s^{\oplus}$. Hence $s^+=s^{\oplus}$.
\end{proof}

\begin{remark} In \cite[Proposition 2.13]{CGH12} restriction categories which admit the structure of a range category are characterized, which implies that restriction semigroups which admit a compatible coEhresmann structure are precisely those all whose elements are {\em open} in the sense of \cite{CGH12}. This notion is defined by means of the Frobenius condition and resembles the definition of an open locale map. Our Proposition \ref{prop:open1} and Proposition \ref{prop:range}(1) are somewhat parallel to this, but involve openness of the range map $r$ of a topological category. The precise connection between the two approaches needs further investigation.
\end{remark}

\begin{example} \label{ex:1} Let $X$ be a set.
\begin{enumerate}
\item The semigroup ${\mathcal{B}}(X)$ of all binary relations on $X$ (with the multiplication of relations from the right to the left) is biEhresmann if one defines
$$
\rho^* = \{(x,x)\colon (y,x)\in \rho \text{ for some } y\in X\}
$$ and
$$
\rho^+ = \{(x,x)\colon (x,y)\in \rho \text{ for some } y\in X\}
$$
to be the support and cosupport relations of $\rho$.
If $|X|\geq 2$, ${\mathcal{B}}(X)$ is neither restriction nor corestriction.
\item The $(2,1,1)$-subalgebra ${\mathcal{PT}}(X)$ of ${\mathcal{B}}(X)$ consists of all partial self-maps of $X$ (by a partial self-map of $X$ we mean a map $Y\to X$ where $Y\subseteq X$ and $Y$ is the {\em domain} of the partial map) is restriction, so it is a range semigroup, while it is not corestriction (if $|X|\geq 2$). Since ${\mathcal{PT}}(X)$ has the identity element, it is a range monoid. 
\item The $(2,1,1)$-subalgebra ${\mathcal I}(X)$ of ${\mathcal{B}}(X)$ consists of all partial injective self-maps of $X$. Since it an inverse monoid, it is  also a birestriction monoid.
\end{enumerate} 
\end{example}
If $X=\{1,2,\dots, n\}$ we write ${\mathcal B}_n$, ${\mathcal{PT}}_n$ and ${\mathcal{I}}_n$ for ${\mathcal B}(X)$, ${\mathcal{PT}}(X)$ and ${\mathcal{I}}(X)$, respectively. 

\begin{example} Let ${\mathbb{X}}$ be a support category (resp. a bisupport category, a range category, a restriction category, a birestriction category or an inverse category), see \cite{CGH12, HL21} for definitions. Then for any object $X$ of ${\mathbb{X}}$ the set ${\mathbb{X}}(X,X)$ of morphism from $X$ to itself forms an Ehresmann monoid (resp. a biEhresmann monoid, a range monoid, a restriction monoid, a birestriction monoid or an inverse monoid). In particular, we have the following.
\begin{enumerate}
\item If ${\mathbb X}$ is the bisupport category of sets and relations between them, then ${\mathbb X}(X,X)$ is isomorphic to the biEhresmann monoid ${\mathcal{B}}(X)$ from Example \ref{ex:1}(1).
\item If ${\mathbb X}$ is the range category of sets and partial maps between them, then ${\mathbb X}(X,X)$ is isomorphic to the range monoid ${\mathcal{PT}}(X)$ from Example \ref{ex:1}(2).
\item If ${\mathbb X}$ is the inverse category of sets and partial injective maps between them, then ${\mathbb X}(X,X)$ is isomorphic to the inverse monoid ${\mathcal{I}}(X)$ from Example \ref{ex:1}(3).
\end{enumerate} 
\end{example}

Further examples can be found, for instance, in \cite[Subsection 2.2]{CGH12}.

\section{Deterministic elements, the compatibility relations and joins} \label{s:determ}
\subsection{Deterministic, codeterministic and bideterministic elements}
The following definition is inspired by \cite{CGH12} where deterministic elements were defined in the context of support categories. 

\begin{definition} \label{def:determ} (Deterministic, codeterministic and bideterministic elements)
Let $S$ be an Ehresmann semigroup. An element $a\in S$ will be called {\em deterministic} if $ea = a(ea)^*$, for all $e\in P(S)$. Codeterministic elements in coEhresmann semigroups are defined dually. An element $s$ of a biEhresmann semigroup will be called {\em bideterministic} if it is both deterministic and codeterministic.
\end{definition}

It is immediate from \eqref{eq:ample_r} that in a restriction semigroup all the elements are deterministic, and dually. 

\begin{lemma}\label{lem:det}
Let $S$ be a biEhresmann semigroup.
An element $s\in S$ is deterministic if and only if $t\leq' s$ implies that $t\leq s$, for all $t\in S$. A dual statement holds for codeterministic elements. Consequently, an element $s$ is bideterministic if and only if
$$
 t\leq s \Longleftrightarrow t\leq' s, \text{ for all } t\in S.  
$$
\end{lemma}

\begin{proof}
If $s$ is deterministic, clearly $t\leq' s$ implies $t\leq s$. If $s\in S$ and $e\in P(S)$ are such that $es = sf$ for some $f\in P(S)$, applying \eqref{eq:axioms_star} and \eqref{eq:rule1r}, we have $s(es)^* = s(sf)^* = s(s^*f)^* = ss^*f = sf = es$, so that $s$ is deterministic. The statement about codeterministic elements holds dually.
\end{proof}

Let ${\mathscr{D}}(S)$ and ${\mathscr{CD}}(S)$  denote the sets of all deterministic  and codeterministic elements of a biEhresmann semigroup $S$, respectively, and put ${\mathscr{BD}}(S) = {\mathscr D}(S) \cap {\mathscr{CD}}(S)$ to be the set of its all bideterministic elements.

It is immediate that in range semigroups all elements are deterministic. 

\begin{proposition}\label{prop:restr1}
Let $S$ be a biEhresmann semigroup.
\begin{enumerate}
\item \label{i33:1} ${\mathscr{BD}}(S)$ is a $(2,1,1)$- subalgebra of $S$ and its projection semilattice coincides with $P(S)$. Moreover, ${\mathscr{BD}}(S)$ is birestriction. 
\item \label{i33:2} ${\mathscr{BD}}(S)$ is the maximum birestriction $(2,1,1)$-subalgebra of $S$ whose projections coincide with $P(S)$.
\end{enumerate}
\end{proposition}

\begin{proof}
(1) Clearly all projections are bideterministic, so projections of  ${\mathscr{BD}}(S)$ coincide with $P(S)$.  
It is easy to see and known that ${\mathscr D}(S)$ 
is closed with respect to the multiplication and the support operation $^*$, and is restriction, (see \cite[Proposition 2.3]{CGH12}). 
A dual statement holds for
${\mathscr{CD}}(S)$, so that ${\mathscr{BD}}(S)$ is a is a $(2,1,1)$-subalgebra of $S$ and is birestriction.

(2) Let $T$ be a birestriction $(2,1,1)$-subalgebra of $S$ with $P(T)=P(S)$ and $a\in T$. It follows from \eqref{eq:ample_r} that $a\in {\mathscr D}(S)$. By symmetry, we also have $a\in {\mathscr{CD}}(S)$. 
\end{proof}

\begin{corollary}\label{cor:det} Let $S$ be a range semigroup. Then ${\mathscr{D}}(S) = S$, ${\mathscr{CD}}(S) = {\mathscr{BD}}(S)$ is a birestriction $(2,1,1)$-subalgebra of $S$  and $P({\mathscr{BD}}(S)) = P(S)$. 
\end{corollary}

\subsection{Partial isomorphisms}  
The following definition is adapted from \cite[Definition 3.6]{CGH12}
where it is given for restriction categories.
\begin{definition} (Partial isomorphisms) \label{def:partial_isom}
An element $s$ of a restriction semigroup $S$ is called a {\em partial isomorphism}, if there is $t\in S$ satisfying $st=t^*$ and $ts=s^*$. 
\end{definition}
An element $t$ as above is called a {\em partial inverse} of $s$\footnote{In \cite{JS01}, the elements $s$ and $t$ as in Definition \ref{def:partial_isom} are said to be {\em true inverses} of each other.}. Clearly, a partial inverse of a partial isomorphism is itself a partial isomorphism. In addition, all projections are partial isomorphisms. Let ${\mathscr{I}}(S)$ denote the set of all partial isomorphisms of a restriction semigroup $S$.

\begin{lemma} \label{lem:true} Let $S$ be a restriction semigroup. 
\begin{enumerate}
\item If $s$ is a partial isomorphism then its partial inverse $s'$ is unique and $(s')'=s$.
\item ${\mathscr{Inv}}(S)$ is an inverse semigroup with $s^{-1} = s'$. Consequently, if $S={\mathscr{I}}(S)$, then $S$ is itself an inverse semigroup.
\end{enumerate}
\end{lemma} 

\begin{proof}
Uniqueness of partial inverses and that they are closed with respect to the multiplication is proved the same as \cite[Lemma 3.7(ii)]{CGH12}. If $s\in {\mathscr{I}}(S)$ then $ss's=ss^* = s$ and $s'ss' = s'(s')^* = s'$, so $s\in {\mathscr{I}}(S)$ is a regular semigroup. Its idempotents are precisely elements of the form $ss'$, so they commute, as $ss' = (s')^*$ and projections of $S$ commute. It follows that ${\mathscr{I}}(S)$ is an inverse semigroup.
\end{proof}

If $S$ is an inverse semigroup, then $S={\mathscr{I}}(S)$ so that $s'=s^{-1}$ for all $s\in S$.

\subsection{Partial isomorphisms vs bideterministic elements in range semigroups}
To compare between partial isomorphisms and bideterministic elements, we need to restrict attention to the setting where both of these notions are defined, which is precisely the class of range semigroups.

\begin{remark}
We show later on in Subsection \ref{subs:groupoidal} that restriction monoids attached to ample groupoids by Garner \cite{G23b} are necessarily range monoids, so both partial isomorphisms and bideterministic elements are defined for these monoids.
\end{remark}

\begin{lemma} \label{lem:6j2}
 Let $S$ be a range semigroup.
\begin{enumerate}
\item If $s\in S$ is a partial isomorphism then $s^+ = (s')^*$.
\item If $s\in S$ is a partial isomorphism then $s$ is bideterministic, i.e., ${\mathscr{I}}(S)\subseteq {\mathscr{BD}}(S)$.
\end{enumerate}
\end{lemma}

\begin{proof}
(1) We have $(s')^*s = ss's = ss^* = s$ and if $fs=s$ for $f\in P(S)$ then also $fss' = ss'$, that is, $f(s')^* = (s')^*$, so that $f\geq (s')^*$. Therefore, $(s')^*$ is the minimum projection $f$ such that $fs=s$. Hence $(s')^*=s^+$.

(2) Since ${\mathscr{I}}(S)$ is inverse, it is birestriction, so every its element is bideterministic. But idempotents of ${\mathscr{I}}(S)$ coincide with projections of $S$, so every element of ${\mathscr{I}}(S)$ is bideterministic.
\end{proof}

\begin{example} \label{ex:24b}
If $S = {\mathcal{PT}}(X)$ then ${\mathscr{BD}}(S) = {\mathscr{I}}(S)={\mathcal{I}}(X)$.
\end{example}

The following example shows that the equality ${\mathscr{I}}(S)={\mathscr{BD}}(S)$ does not hold in general.

\begin{example}
Let $S = \{s\in {\mathcal{PT}}_n\colon s(x) \geq x \text{ for all } x\in {\mathrm{dom}}(s)\}$. This is a range semigroup isomorphic to the semigroup of all lower triangular Boolean matrices over ${\mathbb B}$ such that every column contains at most one $1$. Then ${\mathscr{BD}}(S) = \{s\in {\mathcal{I}}_n\colon s(x) \geq x \text{ for all } x\in {\mathrm{dom}}(s)\}$, which is isomorphic to the semigroup of lower triangular Boolean matrices over ${\mathbb B}$ such that every column and every row contains at most one $1$. On the other hand,
${\mathscr{I}}(S) = P({\mathcal{PT}}_n) = E({\mathcal I}_n)$, which is isomorphic to the semigroup of diagonal Boolean matrices over ${\mathbb B}$. More generally, if $S$ is a subsemigroup of ${\mathcal{PT}}(X)$ which contains all the projections of ${\mathcal{PT}}(X)$ then ${\mathscr{BD}}(S) = S \cap {\mathcal I}(X)$ and
${\mathscr{I}}(S) = \{s\in {\mathscr{BD}}(S)\colon s^{-1}\in S\}$ where $s^{-1}$ denote the inverse partial bijection to $s$. 
\end{example}

\subsection{The compatibility, cocompatibility and bicompatibility relations and joins}\label{subs:comp}
\begin{definition} (Compatible elements)
Let $S$ be a restriction semigroup. Elements $s,t\in S$ are called {\em  compatible} denoted $s\smile t$\footnote{This notation is taken from category theory literature \cite{CG21, CM09, CL24} and is different from the standard notation $\sim$ (used mainly for bicompatibility relation) in semigroup theory literature. The reason for switching to the new notation is that along with the compatibility relation, we will also discuss the cocompatibility and bicompatibility relations, which we will denote by $\frown$ and $\asymp$, respectively.}, if $st^* = ts^*$. 
\end{definition}

\noindent Cocompatible elements in corestriction semigroups are defined dually, denoted by $s\frown t$. We note that the definition of compatible elements in a restriction semigroup is agreed with the definition of compatible parallel maps in a restriction category \cite[Proposition~6.3]{CM09}.

\begin{lemma} \label{lem:2} Let $S$ be a restriction semigroup and $s,t\in S$ be such that $s\smile t$. Then $s\leq t$ if and only if $s^* \leq t^*$. Consequently, $s=t$ if and only if $s^* = t^*$.
\end{lemma}

\begin{proof}
In view of Lemma \ref{lem:23}, only one direction needs proving. Suppose $s\smile t$ and $s^*\leq t^*$. Then $s=ss^* = st^* = ts^* \leq tt^* =t$, as needed.   \end{proof}

The following lemma shows that the  compatibility relation in restriction semigroups is stable with respect to the multiplication from the right and from the left. We will use this fact in the sequel without further mention.

\begin{lemma}\label{lem:compatibility}
Let $S$ be a restriction semigroup and $s,t,u\in S$ be such that $s\smile t$. Then $su\smile tu$ and $us\smile ut$.
\end{lemma}

\begin{proof}
Using \eqref{eq:axioms_star} and \eqref{eq:ample_r}, we have $(su)(tu)^* = su(t^*u)^* = st^*u$.
Similarly, $(tu)(su)^* = ts^*u$. Since $st^*u=st^*u$, we obtain $su \smile tu$.

Furthermore, since $u^*t = t(u^*t)^*\leq t$, we have $(ut)^*=(u^*t)^* \leq t^*$. Therefore,
$(us)(ut)^*=ust^*(ut)^*=
ust^*(us)^*(ut)^*$. Similarly, $(ut)(us)^*=uts^*(ut)^*(us)^*$, and we conclude that $us \smile ut$.
\end{proof}

The following definition is motivated by the definition of a join in restriction categories~\cite{CM09} and in inverse semigroups \cite{Lawson_book}.

\begin{definition}\label{def:join} Let $S$ be a restriction semigroup. For $s,t\in S$ by $s\vee t$ we denote the {\em join} of $s$ and $t$ with respect to $\leq$, if it exists. It is defined as the {\em least upper bound} of $s$ and $t$ with respect to $\leq$, that is, $s,t\leq s\vee t$ and whenever $s,t\leq u$ then $s\vee t\leq u$.
\end{definition}

\begin{lemma} \label{lem:3} If $s, t$ have an upper bound with respect to $\leq$, then $s\smile t$. In particular,
if $s\vee t$ exists, then $s,t\leq s \vee t$, so that $s \smile t$.
\end{lemma}

\begin{proof}
Note that if $s,t \leq u$ then $st^*, ts^*\leq u$. Furthermore, $(st^*)^* = s^*t^* = (ts^*)^*$ by \eqref{eq:rule1r} and \eqref{eq:axioms_star}, so that $st^* = u(st^*)^* = u(ts^*)^* = ts^*$. 
\end{proof}

The following is similar to \cite[Lemma 2.13(2), Lemma 2.15]{KL17}.

\begin{lemma} \label{lem:joins} Let $S$ be a  restriction semigroup.
\begin{enumerate}
\item Let $e,f\in P(S)$ and suppose that $e\vee f$ exists in $S$. Then $e\vee f \in P(S)$ and $e\vee f$ is the join of $e$ and $f$ in $P(S)$.
\item Let $s,t\in S$. Suppose that $s\vee t$ and $s^*\vee t^*$ exist in $S$. Then $(s\vee t)^* = s^* \vee t^*$. 
\item Let $e,f,s\in S$ be such that $e,f\leq s^*$ and suppose that the joins $e\vee f$ and $se\vee sf$ exist in $S$. Then $se\vee sf = s(e\vee f)$.
\end{enumerate}
\end{lemma}

\begin{proof}
(1) We have $e,f \leq e\vee f$, so that $e^*, f^* \leq (e\vee f)^*$. Since $e=e^*$ and $f=f^*$, it follows that $e\vee f\leq  (e\vee f)^*$. But it is easy to see that an element below a projection is  itself a projection.

(2) Since $s,t\leq s\vee t$, we have $s^*, t^* \leq (s\vee t)^*$, so that $s^*\vee t^* \leq (s\vee t)^*$. To prove the reverse inequality, we put $x=(s\vee t)(s^* \vee t^*)$. By Lemma \ref{lem:1}(2) we have $s = ss^* \leq s(s^* \vee t^*) \leq (s\vee t)(s^* \vee t^*) = x$ and similarly $t\leq x$. It follows that $s\vee t \leq x$, so that $(s\vee t)^* \leq x^*$. But because $s^* \vee t^*$ is by part (1) a projection, we have $x^* = ((s\vee t)(s^* \vee t^*))^* = (s\vee t)^*(s^* \vee t^*)\leq s^*\vee t^*$.

(3) Since $se,sf\leq s$, we have $se\vee sf
\leq s$. By part (1) we have that $e\vee f$ is a projection, so $s(e\vee f)\leq s$ as well. It follows that $se\vee sf \smile s(e\vee f)$. In view of Lemma \ref{lem:2} it suffices to prove that $(se\vee sf)^* = (s(e\vee f))^*$. By part (2) we have 
$(se\vee sf)^* = (se)^* \vee (sf)^* = e\vee f$. Since $e, f\leq s^*$, also $e\vee f\leq s^*$, so that $s^*(e\vee f) = e\vee f$.
Now \eqref{eq:ample_r} gives us $(s(e\vee f))^* = s^*(e\vee f) = e\vee f$. The statement follows.
\end{proof}

The statements above can be obviously extended from binary joins to any non-empty finite joins\footnote{The empty join is defined to be equal to $0$.}. 
Dual statements hold for corestriction semigroups, the  cocompatibility relation $\frown$ and the joins with respect to $\leq'$. The following definition is taken from \cite{KL17}.

\begin{definition} (Bicompatible elements)
 Let $S$ be a birestriction semigroup. Elements $s,t\in S$ are called {\em bicompatible}, denoted by $s \asymp t$, if they are both compatible and cocompatible, that is, if $st^* = ts^*$ and $t^+s=s^+t$.   
\end{definition}  

If $S$ is a birestriction semigroup, 
the orders $\leq$ and $\leq'$ coincide, so one does need to specify the order with respect to which the join of $a$ and $b$ is taken. By Lemma \ref{lem:3} and its dual, a necessary condition for elements $a$ and $b$ of a birestriction semigroup $S$ to have a join is $a\asymp b$.

We emphasize that the relations of compatibility, cocompatibility and bicompatibility on a birestriction semigroup in general do not coincide, as the following example illustrates.

\begin{example}\label{ex:23a}
Let $n\geq 2$ and $s,t\in {\mathcal I}_n$. Then $s\smile t$ if and only if $s$ and $t$ are agreed on ${\mathrm{dom}}(s) \cap {\mathrm{dom}}(t)$, that is, for any $x\in {\mathrm{dom}}(s) \cap {\mathrm{dom}}(t)$ we have $s(x) = t(x)$. Similarly, $s\frown t$ if and only if for any $x\in  {\mathrm{ran}}(s) \cap {\mathrm{ran}}(t)$ we have $s^{-1}(x) = t^{-1}(x)$. Take, for example, let $n=2$ and $s=\begin{pmatrix} 1 & 2\\ 1 & \varnothing\end{pmatrix}$, $t=\begin{pmatrix} 1 & 2 \\ \varnothing & 1\end{pmatrix}$ and  $u=\begin{pmatrix} 1 & 2 \\ 2 & \varnothing\end{pmatrix}$ (where the symbol $\varnothing$ in the second line means that the element written above it is not in the domain). These elements are bideterministic (by Example \ref{ex:24b}) and $s\smile t$, $s\not\asymp t$, $s\frown u$, and $s\not\asymp u$. The join $s\vee t$  exists in ${\mathcal{PT}}_2$ and equals $\begin{pmatrix} 1 & 2 \\ 1 & 1\end{pmatrix}$, but note that $s\vee t$ is not bideterministic, so $s\vee t$ does not exist in ${\mathcal I}_2$.
\end{example}

\section{Boolean  restriction and birestriction semigroups}\label{s:brs}
\subsection{Boolean and preBoolean restriction semigroups with local units}\label{subs:brs}
Let $S$ be a restriction semigroup. Recall that an element $0$ of a semigroup is called a {\em left zero} (resp. a {\em right zero}) if $0s = 0$ (resp. $s0=0$) for all elements $s$ of this semigroup.
If $0$ is both a left and a right zero, it is called a {\em zero}. If a semigroup has a zero element, it is unique.

\begin{lemma} \label{lem:11} Suppose $S$ has a left zero element $0$, which is a projection. Then $0$ is the zero and the minimum element of $S$.
\end{lemma}

\begin{proof}
Let $s\in S$. Then $0=0s = s(0s)^* = s0$, so $0$ is the zero.
If $s\in S$, we have $0=s0= s0^*$, so that $0\leq s$.
\end{proof}

A zero element which is a projection is called a {\em restriction zero} in \cite{CL07}, see also \cite{CL24}.

From now on we suppose that the assumption of Lemma \ref{lem:11} is satisfied.
We state our central definition.

\begin{definition} (Boolean and preBoolean restriction semigroups) \label{def:brs} 
Let $S$ be a restriction semigroup with a left zero $0$, which is a projection. It is called {\em Boolean}  if the following conditions are satisfied.
\begin{enumerate}
\item[(BR1)] For any $s,t\in S$ such that $s\smile t$, the join $s\vee t$ exists in $S$.
\item[(BR2)] $(P(S), \leq$) is a generalized Boolean algebra.
\item[(BR3)]  For any $s,t,u\in S$ such that $s\vee t$ exists we have $(s\vee t)u = su \vee tu$.
\end{enumerate}
We call $S$ {\em preBoolean} if it satisfies (BR2), (BR3) and the following relaxation of (BR1):
\begin{enumerate}
\item[(BR1')] for any $s,t\in S$ which have a common upper bound in $S$ their join $s\vee t$ exists in $S$. 
\end{enumerate}
\end{definition}

Let $S$ be a preBoolean restriction semigroup and $e,f\in P(S)$. 
By (BR2) the join of $e$ and $f$ exists in $P(S)$, denote this join by $g$, so that $e$ and $f$ have an upper bound $g$ in $S$. By (BR1') we have that $e\vee f$ exists in $S$. It follows from Lemma \ref{lem:joins}(1) that $e\vee f = g\in P(S)$. Furthermore, since $su, tu \leq (s\vee t)u$, it follows that $su, tu$ have an upper bound and (BR1') implies that the join $su\vee tu$ in (BR3) exists.

In view of Lemma \ref{lem:3}, in a Boolean restriction semigroup the join $s\vee t$ exists if and only if $s\smile t$ while in a preBoolean restriction semigroup elements $s$ and $t$ which are compatible but do not have an upper bound may not have a join.

\begin{remark}\label{rem:def1a}
It is not hard to show that Boolean restriction monoids (i.e., Boolean restriction semigroups which have the identity element $1$) coincide with one-object versions of classical restriction categories of \cite{CM09}.
\end{remark}

\begin{definition} (Local units) We say that a restriction semigroup has {\em local units} if for every $s\in S$ there is $e\in P(S)$ such that $es=s$.
Such an element $e$ is called a {\em left local unit} for $s$.
\end{definition}

Observe that every element $s$ of a restriction semigroup $S$ automatically has a {\em right local unit}, that is, there is $e\in P(S)$ such that $s=se$, one can simply take $e=s^*$.

\begin{remark} If $S$ is a Boolean restriction monoid, $P(S)$ is a unital Boolean algebra and $S$ automatically has  local units. The definition of a Boolean restriction semigroup above then reduces to (an equivalent form of) that of a Boolean  restriction monoid by Garner \cite[Definition 3.2]{G23b}. Hence, Boolean restriction semigroups with local units are a non-unital generalization of Garner's Boolean restriction monoids.
\end{remark}

Condition (BR3) says that the multiplication distributes over  joins from the right, but we now show that also the left distributivity holds (note that (BR3) is not used in the proof).

\begin{lemma}\label{lem:left_distr} Let $S$ be a preBoolean restriction semigroup. Then for any $s,t,u\in S$ where $s\vee t$ exists we have that  $us \vee ut$ exists and $us \vee ut = u(s\vee t)$.
\end{lemma}

\begin{proof} 
Observe that $us, ut \leq u(s\vee t)$ so that $us\vee ut$ exists by (BR1'). Moreover, $us\vee ut\leq u(s\vee t)$.
By Lemma \ref{lem:2} we are left to prove that
\begin{equation}\label{eq:o22a} (u(s\vee t))^* = (us)^*\vee (ut)^*.
\end{equation} We have 
\begin{align*}
(u(s\vee t))^* & = (u^*(s\vee t))^* &  (\text{by } \eqref{eq:axioms_star})\\
& \leq (s\vee t)^*\ & (\text{since }u^*(s\vee t) \leq s\vee t)\\
& = s^* \vee t^* & (\text{by Lemma } \ref{lem:joins}).
\end{align*}
It follows that 
$(u(s\vee t))^* = (u(s\vee t))^*(s^*\vee t^*)$.
Since projections form a generalized Boolean algebra, the latter rewrites to
$(u(s\vee t))^* = (u(s\vee t))^*s^*\vee (u(s\vee t))^*t^*$. Observe that $(u(s\vee t))^*s^* = (u(s\vee t)s^*)^* = (us)^*$ (where for the first equality we used \eqref{eq:rule1r} and for the second equality the fact that $s\leq s\vee t$ which implies that $s=(s\vee t)s^*$). Similarly, $(u(s\vee t))^*t^* = (ut)^*$, and \eqref{eq:o22a} follows.
\end{proof}

\begin{remark} \label{rem:ort}  We call elements $s,t\in S$ {\em orthogonal}, denoted $s \perp t$, if $st^* = ts^*=0$.\footnote{This notion corresponds to the notion of being {\em disjoint} in a restriction category, see \cite{CM09, CL24}.} 
Clearly $s\perp t$ implies $s\smile t$.
One can show that condition (BR1) in the definition of a Boolean restriction semigroup can be replaced by the weaker condition that $a\vee b$ exists whenever $a\perp b$, and condition (BR3) can be replaced by the condition that for any $s,t,u\in S$ where $s\perp t$ we have $(s\vee t)u = su \vee tu$. Analogous relaxations can be made also to the definition of a preBoolean restriction semigroup. We will not use these facts in the sequel and thus omit the details.
\end{remark}

\begin{lemma} Let $S$ be a preBoolean restriction semigroup and $s\in S$. The map $\varphi_s\colon P(S) \to P(S)$, $e \mapsto (es)^*$ preserves finite joins.
\end{lemma}
\begin{proof} If $e,f\in P(S)$,
we have $\varphi_s(e\vee f) = (s(e\vee f))^*$. Applying distributivity and  Lemma \ref{lem:joins}(2), we rewrite the latter as $(se\vee sf)^* = (se)^*\vee (sf)^* = \varphi_s(e) \vee \varphi_s(f)$, and the claim follows.
\end{proof}

\subsection{Boolean birestriction semigroups}
\begin{definition} (Boolean birestriction semigroups)
    Let $S$ be a  birestriction semigroup with a left zero $0$, which is a projection. It is called {\em Boolean} if the following conditions are satisfied.
\begin{enumerate}
\item[(BBR1)] For any two elements $a,b\in S$ such that $a\asymp b$, the join $a\vee b$ exists in $S$.
\item[(BBR2)] $(P(S), \leq$) is a generalized Boolean algebra.
\end{enumerate}

We call $S$ {\em preBoolean} if it satisfies (BBR2), (BBR3) and the following relaxation of (BBR1):
\begin{enumerate}
\item[(BBR1')] for any $s,t\in S$ which have a common upper bound in $S$, their join $s\vee t$ exists in~$S$. 
\end{enumerate}
\end{definition}

Analogously to the proof of Lemma \ref{lem:left_distr} one can prove the following.

\begin{lemma}\label{lem:distr1} Let $S$ be a preBoolean birestriction semigroup and $s,t,u\in S$ be such that $s$ and $t$ have an upper bound. Then the joins $us\vee ut$ and $su \vee tu$ exist in $S$ and $u(s\vee t) = us\vee ut$, $(s\vee t)u = su\vee tu$.
\end{lemma}

\begin{lemma} Every Boolean birestriction semigroup is a preBoolean restriction semigroups with local units.
\end{lemma}

\begin{proof}
We need to show (BR1'). Suppose that $s$ and $t$ have an upper bound in a Boolean birestriction semigroup $S$. Then $s\asymp t$ so that  $s\vee t$ exists by (BBR1). Since $s^+s=s=ss^*$, local units exist. \end{proof}

\begin{example} The symmetric inverse monoid ${\mathcal I}_n$ is a preBoolean restriction monoid and a Boolean birestriction monoid, but is not a Boolean restriction monoid (see Example~\ref{ex:23a}).
\end{example}

\subsection{Morphisms}\label{subs:morphisms_alg} The following definition is motivated by \cite[p. 464, 465]{KL17}.

\begin{definition}\label{def:morphisms1} (Morphisms between preBoolean restriction semigroups) Let $S$, $T$ be preBoolean restriction semigroups. 
A map $f\colon S\to T$ will be called a {\em morphism}, if the following conditions hold.
\begin{enumerate}
\item $f$ is a $(2,1)$-morphism, that is, it preserves the multiplication and the operation $^*$.
\item The restriction of $f$ to $P(S)$ is a proper morphism $f|_{P(S)}\colon P(S)\to P(T)$ between generalized Boolean algebras.
\end{enumerate}
We say that a morphism $f\colon S\to T$ is {\em proper} if for every $t\in T$ there is $n\geq 1$ and $t_1,\dots, t_n\in T$, $s_1,\dots, s_n\in S$, such that $t\leq t_1\vee \dots \vee t_n$ and $f(s_i)\geq t_i$ for all $i\in \{1,\dots,n\}$. We say that $f$ is {\em weakly-meet preserving} if for all $s,t\in S$ and $u\in T$ such that $u\leq f(s), f(t)$ there is $v\leq s,t$ such that $u\leq f(v)$.
\end{definition}

\begin{remark}\label{rem:meets}
 One can show (similarly to \cite[Lemma 8.16]{KL17}) that if $S$ and $T$ are preBoolean restriction semigroups with binary meets with respect to $\leq$, then a morphism $f\colon S\to T$ is weakly meet-preserving if and only if it is meet-preserving in that $f(s\wedge t) = f(s)\wedge f(t)$ for all $s,t\in S$. 
\end{remark}

Morphisms between preBoolean birestriction semigroups (including proper and weakly meet-preserving morphisms) are defined similarly with axiom (1) of Definition \ref{def:morphisms1} replaced by the requirement that $f$ is a $(2,1,1)$-morphism.

\section{\'Etale, co\'etale and bi\'etale topological categories} \label{s:cat_to_sem}
\subsection{Small categories}\label{subs:discr_cat} According to \cite{M71},
a {\em (discrete) small category} $C = (C_1,C_0,u,d,r,m)$ is given by a {\em set of arrows} $C_1$ and a {\em set of objects} $C_0$, together with the structure maps   
$$
u \colon C_0\to C_1, \quad d,r \colon C_1\to C_0, \quad m\colon C_2\to C_1,
$$
 called the {\em unit}, the {\em domain}, the {\em range}, and the {\em multiplication} maps, respectively,
where 
$$
C_2 = \{(x,y)\in C_1\times C_1\colon r(y)=d(x)\}
$$
is the {\em set of composable pairs.} For $(x,y)\in C_2$ we denote $m(x,y)$ simply by $xy$. If $x\in C_0$ we call $u(x)$ the {\em unit} at $x$ and sometimes denote it by $1_x$. The following conditions are required to hold:
\begin{align*}
& (DRU) & d(1_x) & = r (1_x) = x & (\text{domain and range of a unit})\\
& (DP) &  d(xy) & = d(y) & (\text{domain of a product})\\
& (RP) & r(xy) & = r(x) & (\text{range of a product})\\
& (A) & (xy)z & = x(yz) & (\text{associativity})\\
& (UL) & 1_{r(x)}x & = x1_{d(x)} = x & (\text{left and right unit laws})
\end{align*}
It follows form (DRU) that the map $u$ is injective, that is why $C_0$ is sometimes identified with $u(C_0)$, the {\em set of units.} All the categories we consider in the sequel are small. 

Let $C$ and $D$ be categories. 
A {\em functor} $f \colon C \to D$ is a pair of maps $f_1 \colon C_1\to D_1$ and $f_0 \colon C_0\to D_0$ 
that commute with the structure maps of the categories, that is, 
\begin{align*}
 d(f_1(x)) & =  f_0(d(x)), & \text{for all } x\in C_1,\\
r(f_1(x)) & =  f_0(r(x)), & \text{for all } x\in C_1,\\
f_1(xy) & = f_1(x)f_1(y), & \text{for all } (x,y)\in C_2,\\
1_{f_0(x)} & = f_1(1_x), & \text{for all } x\in C_0.
\end{align*}

\subsection{Involutive categories and groupoids} A  category $C$ which is equipped with the additional structure map $i\colon C_1\to C_1$, $a\mapsto a^{-1}$, is called {\em involutive}, if the following conditions hold:
\begin{align*}
d(a^{-1}) & = r(a),  &\text{for all } a\in C_1,\\
(a^{-1})^{-1} & = a, & \text{for all } a\in C_1,\\
(ab)^{-1} & = b^{-1}a^{-1},  & \text{for all } (a,b)\in C_2.
\end{align*}
An involutive category $C= (C_1,C_0,u,d,r,m,i)$ is called a {\em groupoid}, if, in addition, $aa^{-1} = 1_{r(a)}$ and $a^{-1}a = 1_{d(a)}$ for all $a\in C_1$.

\subsection{\'Etale topological categories}
By a {\em topolo\-gi\-cal cate\-gory} $C= (C_1,C_0,u,d,r,m)$ we mean an internal category \cite{M71} in the category of topological spaces. It is given by topological spaces $C_1$ of arrows and $C_0$ of objects, and all the structure maps $u,d,r,m$ are required to be continuous, where $C_2$ is endowed with the relative topology from $C_1\times C_1$.
In the definition of a {\em functor} $f\colon C\to D$ between topological categories it is additionally required that $f_1\colon C_1\to D_1$ and $f_0\colon C_0\to D_0$ are continuous. In the definition of a topological groupoid or a topological involutive category it is additionally required that the inversion map $i$ is continuous.

The following definition is inspired by \cite[Definition 3.1]{Exel08}, see also \cite[Definition 3.1]{St10} and \cite[I.2.8]{R80}.

\begin{definition} (\'Etale, co\'etale and bi\'etale topological categories)
Let $C$ be a topological category such that $C_0$ is a locally compact Hausdorff space. We call $C$  {\em \'etale} if the domain map $d\colon C_1\to C_0$ is a local homeomorphism (also called an {\em \'etale} map). {\em Co\'etale topological categories} are defined dually. We say that $C$ is {\em bi\'etale} if it is both \'etale and co\'etale.
\end{definition}

\begin{remark}\label{rem:et_groupoid} If $C$ is a topological groupoid, it is well known and easily seen that it is \'etale if and only if it is co\'etale if and only if it is bi\'etale.
\end{remark}

From now on we suppose that we are given an \'etale topological category $C$. For co\'etale categories dual statements hold. The following is an adaptation of the respective statement for \'etale groupoids \cite[Proposition 3.2]{Exel08}.

\begin{proposition}
Suppose that $C$ is an \'etale category. Then
the map $u$ is open.
\end{proposition}

\begin{proof}
Let $A\subseteq C_0$ be an open set. We show that $u(A)$ is also open. Since $u(A) = d^{-1}(A) \cap u(C_0)$, it suffices to show that $u(C_0)$ is open. So take an arbitrary $x\in C_0$ and show that $1_x \in u(C_0)$ has a neighborhood which is contained in $u(C_0)$. Since $d$ is a local homeomorphism, there are open sets $A\subseteq C_1$ and $B\subseteq C_0$ such that $1_x\in A$, $x\in B$ and $d\colon A\to B$ is a homeomorphism (where $A$ and $B$ are considered with respect to the relative topology). Then the set $B'=B\cap u^{-1}(A)$ is open in $B$ which implies that $A'=d^{-1}(B')\cap A$ is open in $A$. Observe that $1_x\in A'$ and show that $A' \subseteq u(C_0)$. Suppose $s\in A'$. Then $d(s)\in B'$, so that $1_{d(s)}\in A$. Then $s,1_{d(s)}\in A$ and $d(s) = d(1_{d(s)})$. Since the restriction of $d$ to $A$ is bijective, we conclude that $s=1_{d(s)}\in u(C_0)$, as needed.
\end{proof}

\begin{corollary}\label{cor:u_homeom}
The map $u\colon C_0\to u(C_0)$ is an open and continuous bijection, i.e., a homeomorphism.  
\end{corollary}

\begin{definition} (Local sections, local cosections and local bisections)  Let $C$ be a category.
A subset $A\subseteq C_1$ will be called a {\em local section} if the restriction of the map $d$ to $A$ is injective. {\em Local cosections} are defined dually. If $A$ is a local section and a  local cosection, we call it a {\em local bisection}.
\end{definition}

\begin{definition} (Slices, coslices and bislices) An open local section of a topological category will be called a {\em slice}\footnote{These definitions are motivated by \cite[Definition 3.3]{St10}, but beware that slices of \cite{St10} are what we call bislices, while our slices and coslices are not treated in \cite{St10}.}. {\em Coslices} are defined dually. A {\em bislice} is a subset which is both a slice and a coslice.
\end{definition}

In what follows we work with slices and bislices of \'etale categories. Dual statements hold in co\'etale categories with slices (resp. local sections) replaced by coslices (resp. local cosections). Adopting the notation of \cite{St10}, by $C^{op}$ we denote the sets of all  slices  of a topological category $C$. We denote $\widetilde{C}^{op}$ to be the set of all bislices of $C$. 

Note that a subset of a local section is itself a  local section and an open subset of a slice is itself a slice. 
The following statement is proved similarly to \cite[Proposition 3.5]{Exel08}.

\begin{lemma}\label{lem:basis} \mbox{}
\begin{enumerate}
\item Let $C$ be an \'etale category.
The collection of all slices forms a basis of the topology on $C_1$.
\item Let $C$ be a bi\'etale category. The collection of all bislices forms a basis of the topology on $C_1$.
\end{enumerate}
\end{lemma}

\begin{proof}
(1) Let $V\subseteq C_1$ be an open set and $x\in V$. Since $d$ is a local homeomorphism, there is an open neighborhood $A\subseteq C_1$ of $x$ which is homeomorphic to $d(A)$ via $d$. It follows that $U=V\cap A$ is a slice contained in $V$ and containing $x$. 

(2) By part (1) and its dual statement every slice $V$ is a union of coslices. Since a coslice contained in a slice is a bislice, the statement follows.
\end{proof}

\begin{lemma}\label{lem:loc_sec}
Let $A,B$ be local sections of a category $C$. Then $AB$ is a local section, too.    
\end{lemma}

\begin{proof}
Let $s,t\in AB$ be such that $d(s) = d(t)$. Suppose that $s=a_1b_1$ and $t=a_2b_2$ where $(a_1,b_1), (a_2,b_2)\in (A\times B)\cap C_2$.
Then $d(b_1) = d(s) = d(t) = d(b_2)$. It follows that $b_1=b_2$, as $B$ is a  local section. Hence $d(a_1) = r(b_1) = r(b_2) = d(a_2)$, which similarly yields $a_1=a_2$. Hence $s=t$ and $AB$ is a  local section, as needed.    
\end{proof}

The next proposition is inspired by
\cite[Proposition 2.4]{B23}\footnote{I gratefully acknowledge Mark Lawson for pointing me towards \cite[Proposition 2.4]{B23} and communicating to me a sketch of the proof of Proposition \ref{prop:mult_open} due to Tristan Bice.}.

\begin{proposition} \label{prop:mult_open} Let $C$ be an \'etale category. Then the multiplication map $m$ is open.
\end{proposition}

\begin{proof}
An open set in $C_2$ has the form $O(A,B)=(A\times B)\cap C_2$ where $A,B$ are open in $C_1$. Since $m(O(A,B)) = AB$ and by Lemma \ref{lem:basis}, it suffices to assume that $A$ and $B$ are  slices and prove that $AB$ is open. Note that $AB = A B'$ where $B'=r^{-1}d(A)\cap B$ is a  slice contained in $B$ and such that $r(B')\subseteq d(A)$. So we can assume that $A$ and $B$ are such that $r(B)\subseteq d(A)$ and we prove that $AB$ is open. 

The assumption implies that $d(AB) = d(B)$ and by Lemma \ref{lem:loc_sec} $AB$ is a local section. 
Since $A$ is a slice and a local homeomorphism is an open map, the map $d_A \colon A\to d(A)$, $x\mapsto d(x)$, is a homeomorphism, and so is the map $d_B\colon B\to d(B)$, $x\mapsto d(x)$. Let $\alpha_A\colon d(A)\to A$ and $\alpha_B\colon d(B)\to B$ be their inverse homeomorphisms.
Let $\beta_B\colon B\to d(A)$ be the map given by $\beta_B(x) = r(x)$. Then for $X\subseteq d(A)$ such that $X$ is open in the relative topology of $d(A)$ we have that $X$ is open in $C_0$, so that the set $\beta_B^{-1}(X) = r^{-1}(X) \cap B$ is open in $B$, as $r$ is continuous and $B$ is open. Hence $\beta_B$ is continuous.

The map $m|_{(A\times B)\cap C_2}$, being the restriction of the map $m$ to the open set $(A\times B)\cap C_2$, is continuous. Then the map $\alpha = m|_{(A\times B)\cap C_2} \circ (\alpha_A\circ \beta_B \circ \alpha_B, \alpha_B)\colon d(AB) \to C_1$, which sends $x\in d(AB)$ to the only element $y\in AB$ with $d(y)=x$, is continuous as a composition of continuous maps, and we have $d\alpha (x)= x$ for all $x\in d(AB)$. Let $s\in AB$. Since $d$ is a local homeomorphism and by Lemma \ref{lem:basis}, there is a slice $O$ containing $s$.  It follows that $O\cap d^{-1}\alpha^{-1}(O)$ is a slice. Observe that $O\cap d^{-1}\alpha^{-1}(O) \subseteq AB$. Indeed, if $x\in O \cap d^{-1}\alpha^{-1}(O)$, we have $x=\alpha d(y)$ for some $y\in O\cap AB$ where $d(x)= d(y)$. It follows that $x=y$, as $O$ is a slice. Hence $O \cap d^{-1}\alpha^{-1}(O)$ is a neighborhood of $s$, contained in $AB$. This proves that $AB$ is open.
\end{proof}

\begin{lemma} \label{lem:ud} Suppose that the category $C$ is  \'etale and $A\subseteq C_1$ is a slice.
Then $A$ is homeomorphic to $1_{d(A)}$.
\end{lemma}

\begin{proof} The statement follows from the fact that a bijective open and continuous map is a homeomorphism.
\end{proof}

Let $C$ be an \'etale category. If $A,B\in C^{op}$, we put
$$
AB = \{ab\colon a\in A, b\in B, (a,b)\in C_2\} \, \text{ and }\,  A^* = ud(A) = 1_{d(A)}.
$$ 

The following statement is inspired by \cite[Proposition 3.4]{St10}.

\begin{proposition} \label{prop:restriction} Let $C$ be an \'etale category. 
The set $C^{op}$ is closed with respect to the multiplication and the unary operation $^*$. In addition, $u(C_0)$ is a slice, and $C^{op}$ with respect to the operations of the multiplication and $^*$ forms a restriction monoid with the unit $u(C_0)$.
\end{proposition}

\begin{proof}
Let $A,B\in C^{op}$. By Lemma \ref{lem:loc_sec} we have that $AB$ is a local section and because the multiplication map is open, it follows that $AB\in C^{op}$. 

As $du = id_{C_0}$, we have that $d(a) = d(b)$ if and only if $d(ud(a)) = d (ud(b))$. It follows that for $A\in C^{op}$ the map $d\colon ud(A)\to d(A)$ is injective. Moreover, as $d$ and $u$ are open maps, $A^* = ud(A)$ is open, hence a slice. Since $u(C_0)$ is open and  $u(a) = u(b)$ implies $a=du(a) = du(b) =b$, it follows that $u(C_0)$ a slice.

We now verify the axioms of a restriction monoid. Associativity of the multiplication easily follows from the associativity of the multiplication in the category $C$. Clearly, $u(C_0)$ is a unit element of $C^{op}$. Furthermore, if $A\in C^{op}$, we have $AA^* = \{a 1_{d(a)} \colon a\in A\} = A$, applying axiom (UL). Since $A^*B^* = 1_{d(A)}1_{d(B)} = 1_{d(A)\cap d(B)} = 1_{d(B)}1_{d(A)}$, the identity $A^*B^* = B^*A^*$ holds.

Let $A,B\in C^{op}$, we need to show that $(AB)^* = (A^*B)^*$. Note that $(AB)^*$ is the set of all $ud(ab)$ where $(a,b)$ runs through $(A\times B) \cap C_2$. But for such an $(a,b)$ we have $d(ab) = d(1_{d(a)}b)$. Since  $\{1_{d(a)}b\colon (a,b)\in (A\times B) \cap C_2\} = A^*B$, the needed equality follows.

Finally we show the identity $A^*B = B(AB)^*$. As noted above, we have $A^*B = \{ 1_{d(a)}b \colon (a,b)\in (A\times B)\cap C_2\} = \{b\in B \colon r(b) \in d(A)\}.$ On the other hand, $(AB)^* = \{1_{d(ab)} \colon (a,b)\in (A\times B)\cap C_2\} = \{1_{d(b)} \colon r(b)\in d(A)\}$. Since the map $B\to B^*$, $b\mapsto 1_{d(b)}$, is bijective, it follows that $B(AB)^* = \{b\in B \colon r(b) \in d(A)\} = A^*B$, as needed.
\end{proof}

\begin{remark}\label{rem:sp}
It is easy to see that projections of $C^{op}$ are precisely the slices contained in $u(C^0)$ which are precisely the opens of $u(C^0)$. It follows that they form a spatial frame. Moreover, any compatible family of elements of $C^{op}$ has a join. Hence it is natural to call $C^{op}$ a {\em spatial complete restriction monoid.} 
\end{remark}

\section{Ample topological categories and their Boolean restriction semigroups}\label{s:ample}

\subsection{Ample, coample and biample topological categories}
The following definition is motivated by \cite[Definition 3.5, Proposition 3.6]{St10}.

\begin{definition} (Ample, coample and biample categories)
An \'etale category will be called {\em ample} if compact slices form a basis of the topology of $C_1$. {\em Coample} categories are defined dually. A bi\'etale category $C$ will be called {\em biample} if compact bislices form a basis of the topology of $C_1$.
\end{definition}

\begin{lemma}\label{lem:coample}
A bi\'etale category is biample  if and only if it is both ample and coample.
\end{lemma}

\begin{proof}
Since a bislice is a slice and a coslice, a biample category is both ample and coample. For the other way around, suppose $C$ is both ample and coample. It suffices to show that every compact slice is a union of compact bislices. But compact coslices form a basis of the topology. So every compact slice is a union of compact coslices contained in it. Since a coslice contained in a slice is a bislice, the statement follows.   
\end{proof}

\begin{proposition}\label{prop:char_etale}\mbox{}
\begin{enumerate}
    \item An \'etale category $C$ is ample if and only if $C_0$ is a locally compact Stone space.
    \item A bi\'etale category is ample if and only if it biample.
\end{enumerate}
\end{proposition}

\begin{proof}
(1) If $C$ is ample then  compact slices contained in $u(C_0)$ form a basic of the topology on $u(C_0)$. Note that compact slices contained in $u(C_0)$ are of the form $u(A)$ where $A$ is a compact-open subset of $C_0$. It follows that $C_0$ is a Hausdorff space such that compact-open sets form a basis of its topology. That is, $C_0$ is a locally compact Stone space. Conversely, suppose that $C$ is \'etale and compact-open sets form a basis of the topology on $C_0$. Since $C^{op}$ is the basis of the topology on $C_1$, it suffices to show that every slice $U\in C^{op}$ is a union of compact slices. But $d(U)$ is a union of compact-open sets in $C_0$. Applying the fact that $U$ is homeomorphic to $d(U)$ via $d|_U$, the result follows.

(2) Let $C$ be bi\'etale and ample. By (1) $C_0$ is a locally compact Stone space. By the dual of (1) $C$ is coample and thus biample, by Lemma \ref{lem:coample}. The reverse implication is clear, since a bislice is a slice.
\end{proof}

Bearing in mind Remark \ref{rem:et_groupoid}, we have the following corollary.

\begin{corollary} \label{cor:ample_biample} An \'etale groupoid is ample if and only if it is biample. 
\end{corollary}

\subsection{The restriction semigroup of slices}
Adopting the notation of \cite{St10}, we denote the set of compact slices of an ample category $C$ by $C^a$.  The set of all compact bislices of $C$ will be denoted by $\widetilde{C}^a$.

\begin{proposition}
Let $C$ be an ample category. Then $C^a$ is closed with respect to the multiplication and the operation $^*$ and thus forms a $(2,1)$-subalgebra of the restriction monoid $C^{op}$.   
\end{proposition}

\begin{proof} If $A,B$ are compact slices then $AB$ is a slice and it is compact since $m$ is open and the image of a compact set under an open map is compact. Similarly, if $A$ is a compact slice, so is $A^* = ud(A)$.
\end{proof}

Observe that $u(C_0)$ is compact if and only if $C_0$ is compact. Therefore,  $C^a$ is a monoid if and only $C_0$ is a compact space.

\begin{proposition}\label{prop:Boolean}
Let $C$ be an ample category. Then $C^a$ is a Boolean  restriction semigroup with local units. 
\end{proposition}

\begin{proof}
It is easy to see that $A\in C^a$ satisfies $A=A^*$ if and only if $A\subseteq u(C_0)$. It follows that projections of $C^a$ are all compact-open subsets of the locally compact Stone space $u(C_0)$ and thus form a generalized Boolean algebra, so that (BR1) holds. Furthermore, suppose $A\smile B$ in $C^a$. We show that $A\cup B \in C^a$. Let $a,b\in A\cup B$ be such that $d(a) = d(b)$. Since $a1_{d(b)} = a1_{d(a)} = a$ and similarly $b1_{d(a)} = b$, we have that $a,b \in AB^* = BA^*$. Since $AB^*$ is a slice, we conclude that $a=b$. It follows that $A\vee B$ exists in $C^a$ and equals $A\cup B$, so that (BR2) holds, too. For (BR3), let $A,B,C \in C^a$ and $A\smile B$. Then $AC, BC \subseteq (A\cup B)C$ so that $AC\cup BC \subseteq (A\cup B)C$. For the reverse inclusion, let $x\in (A\cup B)C$. Then $x=st$, where $s\in A\cup B$, $t\in C$, so that $s\in A$ or $s\in B$ and $x\in AC\cup BC$. 

It remains to prove that $C^a$ has local units. Let $A\in C^a$ and $a\in A$. There is a compact neighborhood $E_{r(a)} \subseteq C_0$ of $r(a)$. Then $1_{E_{r(a)}}A$ is a compact slice which contains $a$ and is contained in $A$. We have that $A=\cup_{a\in A} 1_{E_{r(a)}}A$ is an open cover of a compact set $A$, so there is a finite subcover $A=\cup_{a\in I} 1_{E_{r(a)}}A$. Since  $E=\cup_{a\in I} E_{r(a)}$ is a compact-open set, $1_E$ is a compact slice, so that $A=1_{E}A\in C^a$, and  $C^a$ has local units. 
\end{proof}

\subsection{The case where $r$ is open} In the sequel ample categories such that the range map $r$ is open will play an important role, so we give them a name.

\begin{definition} (Strongly \'etale and strongly ample topological categories) We call an \'etale (resp. ample) topological category {\em strongly \'etale}  (resp. {\em strongly ample}) if the map $r$ is open. Strongly co\'etale and strongly coample topological categories are defined dually.
\end{definition}

The definition implies that the class of strongly ample topological categories is an intermediate class between those of ample and biample topological categories, and similarly for the class of strongly \'etale topological categories. Corollary~\ref{cor:ample_biample} implies that for a topological groupoid the notions of being ample, coample, strongly ample, strongly coample or biample are equivalent. Likewise, Remark~\ref{rem:et_groupoid} yields that for a topological groupoid the notions of being \'etale, co\'etale, strongly \'etale, strongly co\'etale or bi\'etale are equivalent. 

Let $C$ be a strongly \'etale topological category. For each $A\in C^{op}$ there is a well defined element $A^+ = ur(A) = 1_{r(A)}\in C^{op}$. This gives rise to a unary operation $A\mapsto A^+$ on $C^{op}$.

\begin{definition} (Boolean range semigroups)
A {\em Boolean range semigroup} is a range semigroup $(S, \cdot, ^*, ^+)$ such that $(S,\cdot, ^*)$ is a Boolean restriction semigroup. 
\end{definition}

\begin{proposition}\mbox{} \label{prop:range}
\begin{enumerate}
\item Let $C$ be a strongly \'etale category. Then $(C^{op}, \cdot \, , ^*, ^+)$ is a range monoid. 
\item Let $C$ be a strongly ample category. Then $(C^a, \cdot \,,  ^*, ^+)$ is a Boolean range semigroup.
\end{enumerate}
\end{proposition}

\begin{proof}
(1)  The identities $A^+A = A$ and $(AB)^+ = (AB^+)^+$ are verified similarly as the identities $AA^* = A$ and $(AB)^* = (A^*B)^*$ in the proof of Proposition \ref{prop:restriction}. The identities $A^+B^+ = B^+A^+$, $(A^+)^* = A^+$ and $(A^*)^+ = A^*$ are easy to verify. This, together with Proposition \ref{prop:restriction}, yields the claim.

(2) Let $A\in C^a$ and let $\alpha_A\colon d(A)\to A$ be the inverse homeomorphism to $d|_A\colon A\to d(A)$. Then the map $r\alpha_A\colon d(A) \to r(A)$ is continuous, so it takes the compact set $d(A)$ to the compact set $r(A)$. It follows that $A^+ = 1_{r(A)}$ is compact-open, so it is in $C^a$. The statement now follows from part (1) and Proposition \ref{prop:Boolean}.
\end{proof}

\begin{proposition}\mbox{} \label{prop:det}
\begin{enumerate}
\item
Suppose $C$ is a strongly \'etale category. Then the set of bideterministic elements of the range monoid $C^{op}$ coincides with the set $\widetilde{C}^{op}$ of bislices. 
\item Suppose $C$ is a strongly ample category. Then the set of bideterministic elements of the Boolean range semigroup $C^{a}$ coincides with the set $\widetilde{C}^a$ of compact bislices. Moreover, $\widetilde{C}^a$ is a Boolean birestriction semigroup.
\end{enumerate}
\end{proposition}

\begin{proof}
(1) From Proposition \ref{prop:range} we know that $(C^{op}, \cdot \, , ^*, ^+)$ is a range monoid. By Corollary \ref{cor:det} its bideterministic elements are precisely its codeterministic elements. Let $A\in C^{op}$ be a bislice and $E$ be an open set in $C_0$. We show that $A1_E = 1_{r(A1_E)}A$. Let $x\in A1_E$. Then $x\in A$ and $1_{r(x)}\in 1_{r(A1_E)}$ which implies the inclusion $A1_E \subseteq 1_{r(A1_E)}A$. For the reverse inclusion, suppose $x\in 1_{r(A1_E)}A$. This means that $x\in A$ and $r(x) \in  r(A1_E)$. Since the restriction of $r$ to $A$ is injective, it follows that $x \in A1_E$, so that $A$ is  codeterministic.
Suppose now that $A\in C^{op}$ is codeterministic and show that $r|_A$ is injective. From the converse, suppose that $a,b\in A$ are such that $r(a)=r(b)$ and $a\neq b$. Let $O_a$ and $O_b$ be disjoint neighborhoods of $a$ and $b$, respectively. Then $O_aA$ contains $a$ and avoids $b$ and $O_bA$ contains $b$ and avoids $a$. On the other hand, for every open set $E\subseteq C_0$ we have that $A1_{E}$ either contains both $a$ and $b$ (which is the case when $r(a)=r(b)\in E$) or contains none of them. Hence $A$ is not codeterministic, which is a contradiction. This proves that $A$ is a coslice and thus a bislice.

(2) follows from (1) as bideterministic elements of $C^{a}$ are precisely bideterministic elements of $C^{op}$ which are compact. By Proposition \ref{prop:restr1}, $\widetilde{C}^a$ is a birestriction semigroup. Its projections $P(\widetilde{C}^a) = P(C^a)$ form a generalized Boolean algebra. If $A,B\in \widetilde{C}^a$ and $A\asymp B$ then $A\smile B$, so $A\vee B$ exists in $C^a$ by Proposition \ref{prop:range}(2). Since $A\vee B$ is just the union of $A$ and $B$, and the union of two compact sets is compact, $A\vee B\in \widetilde{C}^a$. Then we get that $A\vee B$ exists in $\widetilde{C}^a$ and equals $A\cup B$.
\end{proof}

The following is similar to Remark \ref{rem:sp}.

\begin{remark}\label{rem:sp1}
Since range semigroups are biEhresmann, Proposition \ref{prop:restr1} implies that $\widetilde{C}^{op}$ is a birestriction monoid. Since $P(\widetilde{C}^{op}) = P(C^{op})$,  projections of $\widetilde{C}^{op}$ form a spatial frame. Moreover, any compatible family of elements of $\widetilde{C}^{op}$ has a join in $C^{op}$ and belongs to $\widetilde{C}^{op}$, so it has a join in $\widetilde{C}^{op}$. It follows that $\widetilde{C}^{op}$ is a {\em spatial complete birestriction monoid.} These monoids were studied in \cite{KL17}\footnote{They were called complete restriction monoids in \cite{KL17}.}.
\end{remark}

\subsection{The case where $r$ is a local homeomorphism}
Let $S$ be a Boolean range semigroup and suppose that $s,t\in {\mathscr{BD}}(S)$ satisfy $s\smile t$. Example \ref{ex:23a} shows that $s\vee t$ may exist in $S$, but not in ${\mathscr{BD}}(S)$. This motivates to single out the class of Boolean range semigroups consisting of elements which are join-generated by bideterministic elements.

\begin{definition}\label{def:etale} (\'Etale Boolean range semigroups\footnote{This terminology is consistent with the term \'etale Ehresmann quantal frame of \cite[p. 407]{KL17}.}) We say that a Boolean range semigroup is {\em \'etale} if everyone of its elements $s$ can be written as a finite join $s=s_1\vee \cdots \vee s_n$ of (compatible but not necessarily bicompatible) bideterministic elements $s_i$.
\end{definition}

\begin{proposition} \label{prop:etale} Let $C$ be a biample category. Then the Boolean range semigroup $C^a$ is \'etale.
\end{proposition}

\begin{proof}
Since compact bislices form a basis of the topology on $C_1$, every compact slice is a finite union of compact bislices, which, by Proposition \ref{prop:det}, are precisely the bideterministic elements of $C^a$.
\end{proof}

\subsection{Morphisms: cofunctors} The following definition extends that of the action of a topological groupoid on a topological space \cite{BEM12}.

\begin{definition} (Action of a topological category on a topological space)
Let $C$ be a topological category and $X$ a topological space. 
An {\em action} of $C$ on $X$ is a pair $(\mu, f)$ where $f\colon X\to C_0$ is a continuous map, which, following \cite{BEM12}, we call the {\em anchor map} and $\mu\colon C_1\times_{d,f} X = \{(s,x)\in C_1\times X\colon d(s)= f(x)\} \to X$, $(s,x)\mapsto s\cdot x$, is a continuous map such that:
\begin{enumerate}
\item[(A1)] $f(s\cdot x) = r(s)$ for all $(s,x)\in C_1\times_{d,f} X$;
\item[(A2)] $s\cdot (t\cdot x) = st\cdot x$ for all $s,t\in C_1$ and $x\in X$ for which both sides of the equality are defined;
\item[(A3)] $1_{f(x)}\cdot x = x$ for all $x\in X$.
\end{enumerate}
We say that the action $(\mu,f)$ is {\em injective} if for all distinct $x,y\in X$ and $s\in C_1$ with $d(s)=f(x)=f(y)$ we have that $s\cdot x\neq s\cdot y$ or, equivalently, for all $x,y\in X$ and $s\in C_1$ satisfying $d(s)=f(x)=f(y)$ and $s\cdot x = s\cdot y$, we have that $x=y$.
\end{definition}
If $C$ is a groupoid then $(\mu,f)$ is automatically injective: if  $x,y\in X$ and $s\in C_1$ are such that $d(s)=f(x)=f(y)$ and $s\cdot x = s\cdot y$, we have 
$$x=1_{f(x)}\cdot x = s^{-1}s\cdot x = s^{-1}\cdot (s\cdot x)=s^{-1}\cdot (s\cdot y) = s^{-1}s\cdot y = 1_{f(y)}\cdot y=y.$$   

If an action $(\mu, f)$ of $C$ on $X$ is given, we define the {\em transformation category} $C\rtimes X$, as follows. Its space of objects is $X$ and its space of arrows is $C_1\times_{d,f} X$. The structure maps are given by $u(x) = (1_{f(x)},x)$, $d(s,x) = x$, $r(s,x) = s\cdot x$ and $(s,t\cdot x) (t,x) = (st,x)$. 

The following definition is inspired by \cite{BEM12}, see also \cite{CG21, K12}.

\begin{definition} (Cofunctors)  A {\em cofunctor} $F=(\mu, f, \rho)\colon C\rightsquigarrow D$ between ample topological categories $C$ and $D$ is given by an action $(\mu, f)$ of $C$ on $D_0$ with $f$ being proper and a functor $\rho\colon C\rtimes D_0 \to D$ between topological categories which acts identically on objects such that for every compact slice $A\subseteq C_1$ we have that $F_*(A) = \{\rho_1(s,x)\colon s\in A, (s,x)\in C\rtimes D_0\}$ is compact-open.
\end{definition}

It is easy to see that $F_*(A)$ is a compact slice for every compact slice $A$. 

Cofunctors between ample topological categories compose as follows. Suppose that $C$, $D$ and $E$ are ample topological categories, and let $F = (\bar{\mu}, \bar{f}, \bar{\rho}) \colon C\rightsquigarrow D$ and $G = (\tilde{\mu}, \tilde{f}, \tilde{\rho})\colon D \rightsquigarrow E$ be cofunctors. The composition $G\circ F \colon C \rightsquigarrow E$ is defined as $G\circ F = (\mu, f, \rho)$ where $f = \bar{f}\tilde{f}\colon E_0\to C_0$, the map $\mu \colon C_1\times_{d, f} E_0 \to E_0$, $(s,x)\mapsto s\cdot x$ is given by $s\cdot x = \tilde{\mu}(\bar{\rho}_1(s, \tilde{f}(x)),x)$, and the functor $\rho\colon C\rtimes E_0 \to E$ is given by $\rho_0(x)=x$ and $\rho_1(s,x) = \tilde{\rho}_1(\bar{\rho}_1(s,\tilde{f}(x)),x)$. It is routine to veriry that $G\circ F$ is well defined and is a cofunctor.

We now define special kinds of cofunctors. 

\begin{definition}
Let $F=(\mu,f,\rho)\colon C\rightsquigarrow D$ be a cofunctor. We say that $F$ is:
\begin{itemize}
\item {\em injective on arrows} if for all $x\in X$ we have that $\rho_1(s,x) = \rho_1(t,x)$ where $(s,x),(t,x)\in C_1\times_{d,f} D_0$ implies that $s=t$;
\item {\em surjective on arrows}
if for all $x\in X$ and $t\in D_1$ with $d(t) = x$ there is $(s,x) \in C_1\times_{d,f} D_0$ with $\rho_1(s,x) = t$;
\item {\em bijective on arrows} if it is both injective and surjecrive on arrows.
\end{itemize}
\end{definition}

Let $C,D$ be ample categories and $F = (\mu, f, \rho)\colon C\rightsquigarrow D$ a cofunctor. It is easy to see that $F_* \colon C^a \to D^a$ is a $(2,1)$- morphism. Moreover, since $f$ is proper, the restriction of $F_*$ to projections is a proper morphism of generalized Boolean algebras, by Theorem \ref{th:Stone_classical}.

\begin{proposition}\label{prop:morphisms1} Let $F=(\mu,f,\rho)\colon C\rightsquigarrow D$ be a cofunctor between ample categories.
\begin{enumerate}
\item If $F$ is injective on arrows then $F_*\colon C^a\to D^a$ is a weakly meet-preserving morphism.
\item If $F$ is surjective on arrows then $F_*\colon C^a\to D^a$ is a proper morphism.
\item If the action $(\mu,f)$ is injective then $F_*$ preserves bideterministic elements.
\end{enumerate}   
\end{proposition}

\begin{proof}
(1) Suppose that $F$ is injective on arrows. Let $T\in D^a$ be such that $T\subseteq F_*(A)\cap F_*(B)$ for $A,B\in C^a$. Then every $t\in T$ is of the form $t=\rho_1(s,d(t)) = \rho_1(s',d(t))$ where $s\in A$, $s'\in B$ and $d(s) = d(s') = f(d(t))$. Since $F$ is injective on arrows, we have $s = s' \in A\cap B$ and it is uniquely determined by $t$. Then there is a neighborhood $S_t\in C^a$ of $s$, contained in $A\cap B$. Then $t\in F_*(S_t)$ and we have that $T = \cup_{t\in T} (F_*(S_t)\cap T)$. Since $T$ is compact, there are $t_1,\dots, t_n\in T$ such that $T=\cup_{i=1}^n (F_*(S_{t_i})\cap T)$ so that $T\subseteq F_*(\cup_{i=1}^n S_{t_i})$. Since $\cup_{i=1}^n S_{t_i}$ is a finite compatible join of elements of $C^a$, it itself belongs to $C^a$, so we have shown that $F_*$ is weakly meet-preserving.

(2) Suppose that $F$ is surjective on arrows and 
let $T\in D^a$. For each $t\in T$ there is some $s\in S$ such that $t=\rho_1(s,d(t))$. Let $S_t\in C^a$ be a neighborhood of $s$. Then $T \subseteq \cup_{t\in T}F_*(S_t)$. Compactness of $T$ ensures that there are $t_1,\dots, t_n\in T$ such that $T \subseteq \cup_{i=1}^nF_*(S_{t_i})$. Denote $T_i = F_*(S_{t_i})\cap T$. Then $T= \cup_{i=1}^n T_i$. Since every $T_i$ is open, it is a union of compact slices $T_i = \cup_{j\in J_i} T_{ij}$ We have an open cover  $T = \cup_{i=1}^n \cup_{j\in J_i} T_{ij}$ of the compact set $T$. Then there is a finite subcover $T = T_1' \cup \dots \cup T_m'$. Note that the compact slices $T_1', \dots, T_m'$ are bounded from the above by $T$, so they are compatible. Since $F_*(S_i) \geq T_{ij}$ for all $j\in J_i$, every $T_k'$ is contained in some $F_*(S_i)$, which shows that $F_*$ is proper, as needed.

(3) Suppose now that $\mu$ is injective and that $T\in \widetilde{C}^a$. As $F_*(T)\in D^a$, in view of Proposition \ref{prop:det} it suffices to show that $F_*(T)$ is a coslice. Suppose that $s,s'\in F_*(T)$ and $r(s) = r(s')$. Then $s=\rho_1(t,x)$ and $s' = \rho_1(t',x')$ for some $t,t'\in T$ where $d(t)=f(x)=d(t') = f(x')$ and $r(s) = t\cdot x = t'\cdot x'$. Since $T$ is a slice and $d(t)=d(t')$, we have $t=t'$. Since $\mu$ is injective, $x=x'$, so that $s=s'$, as needed.
\end{proof}

\subsection{Proper and continuous covering functors} \label{subs:functors}
We now show that cofunctors $F=(\mu, f, \rho)\colon C\rightsquigarrow D$ which are bijective on arrows can be equivalently described as certain functors $D\to C$ called {\em proper and continuous covering functors} (inspired by similar notions in \cite{LL13,KL17}). We first define these functors.

\begin{definition} (Covering functors)
Let $C$ and $D$ be (discrete) categories and $f=(f_0, f_1)$ a functor from $C$ to $D$. We say that $f$ is {\em star injective} if $f_1(x) = f_1(y)$ and $d(x) = d(y)$ imply that $x=y$. It is {\em star surjective} if for any $y\in D_1$ where $d(y)$ is in the image of  $f_0$ there is $x\in C_1$ with $f_1(x)=y$. We say that $f$ is {\em star bijective} or a {\em covering functor} if it is both star injective and star surjective.  
\end{definition}

\begin{definition} (Continuous and proper functors)
If $C$ and $D$ are topological categories, then a functor $f=(f_0, f_1)$ from $C$ to $D$ is  
said to be {\em continuous} if it is a functor between topological categories, that is, if $f_0$ and $f_1$ are continuous maps. We say that $f$ is {\em proper} if $f_1$ is proper.
\end{definition}

Observe that if $f_1$ is proper then $f_0$ is also proper. Indeed, let $A\subseteq D_0$ be a compact set. Then $f_0^{-1} (A) = d(f_1^{-1}(1_A))$ which is compact.

Let $C$ and $D$ be ample categories and $F=(\mu, f_0, \rho)\colon C\rightsquigarrow D$
 a cofunctor which is bijective on arrows. For every $t\in D_1$ then there is a unique $(s,x)\in C_1\times_{d,f} D_0$ with $\rho_1(s,x) = t$. We put $f_1(t)=s$. It is routine to verify that $f=(f_1,f_0)\colon D\to C$ is a continuous covering functor. Moreover, since $F_*$ takes compact slices to compact  slices, so does $f_1^{-1}$, so $f_1$ is a proper map. 

For the other way around, let $f=(f_1,f_0)\colon D\to C$ be a proper and continuous covering functor. We define the action $\mu$ of $C$ on $D_0$ as follows. Let $s\in C_1$ and $x\in D_0$ be such that $d(s) = f(x)$. By star-bijectivity, there is a unique $t\in D_1$ with $d(t)=x$ and $f_1(t)=s$. We put $s\cdot x = r(t)$. It is routine to verify that this defines a continuous action. Furthermore, we define $\rho = (\rho_1,\rho_0)\colon C\rtimes D_0 \to D$ such that  $\rho_0$ is the identity map on $D_0$ and $\rho_1(s,x) = t$. It is routine to check that $\rho$ is a functor. We set $F=(\mu, f_0, \rho)$. If $A\in C^a$ then $F_*(A) = f_1^{-1}(A)\in D^a$, so that $F$ is indeed a cofunctor. The construction implies that it is also bijective on arrows. 

\begin{remark} \label{rem:morphisms} If we drop the requirement that $F$ is bijective on arrows, $f_1$ is no longer a function, but rather a {\em relation}. Then we translate $F\colon C\rightsquigarrow D$ into a {\em proper and lower semicontinuous relational covering morphism} from $D$ to $C$. If $F$ is injective (resp. surjective) on arrows, the corresponding relational covering morphism turns out to be {\em at most single valued} (resp. {\em at least single valued}). Details on this approach to morphisms can be developed following the lines of \cite{KL17,K21}.
\end{remark}

\section{From a preBoolean restriction semigroup with local units to an ample category}\label{s:sem_to_cat}
\subsection{The category of germs of a  preBoolean restriction semigroup with local units}\label{subs:sem_cat}
Throughout this section, $S$ denotes a  preBoolean restriction semigroup with local units. Recall that $\widehat{P(S)}$ is the space of prime characters of the generalized Boolean algebra $P(S)$.  Let $s\in S$ and $\varphi \in \widehat{P(S)}$ be such that $\varphi(s^*)=1$. Define the map 
\begin{equation*}\label{eq:circ}
s\circ \varphi\colon P(S) \to {\mathbb B}, \,\, (s\circ\varphi) (e) = \varphi ((es)^*), \,\, e\in P(S).
\end{equation*}

\begin{lemma} \label{lem:prime}
The map $s\circ \varphi$ is a prime character.
\end{lemma}

\begin{proof} 
Since $S$ has local units, there is $e\in P(S)$ such that $es=s$. Then  $$(s\circ \varphi)(e) = \varphi((es)^*) = \varphi(s^*) =1,$$ so that the map $s\circ \varphi$ is non-zero. For any $e,f\in P(S)$ we have 
\begin{align*}
 (s\circ \varphi)(e) (s\circ \varphi)(f)   & = \varphi((es)^*)\varphi((fs)^*) & (\text{by the definition of } s\circ \varphi)\\
 & = \varphi((es)^*(fs)^*) & (\text{since } \varphi \text{ preserves the multiplication})\\
 & = \varphi(es(fs)^*) & (\text{by } \eqref{eq:axioms_star})\\
 & = \varphi(efs)^* & (\text{by } \eqref{eq:ample_r})\\
 & = (s\circ \varphi) (ef) & (\text{by the definition of } s\circ \varphi),
\end{align*}
so that $s\circ \varphi$ preserves the multiplication which is the meet operation in $P(S)$. 
It follows that $s\circ \varphi$ is monotonous, so that $(s\circ \varphi)(e\vee f) \geq (s\circ \varphi)(e)\vee (s\circ \varphi)(f)$ for all $e,f\in P(S)$. To show that $s\circ \varphi$ preserves joins, it suffices to suppose that $(s\circ\varphi) (e \vee f) = 1$ and show that $(s\circ\varphi)(e) = 1$ or $(s\circ\varphi)(f) = 1$. Since $es, fs\leq s$, the join $es\vee fs$ exists by (BR1'). Applying (BR3), Lemma \ref{lem:joins} and the fact that $\varphi$ preserves joins, we obtain: $$1=\varphi(((e\vee f)s)^*) = \varphi((es\vee fs)^*) = \varphi ((es)^* \vee (fs)^*) = \varphi((es)^*) \vee \varphi((fs)^*).$$
Hence either
$\varphi((es)^*)=1$ or $\varphi((fs)^*)=1$, that is, $(s\circ\varphi)(e) = 1$ or $(s\circ\varphi)(f) = 1$, as needed.
\end{proof}

Following Exel \cite{Exel08}, we set
$$
\Omega = \{(s,\varphi)\in S\times \widehat{P(S)}\colon \varphi\in D_{s^*}\}.
$$

For every $(s,\varphi), (t,\psi)\in \Omega$, we say that $(s,\varphi)\sim (t,\psi)$ if $\varphi=\psi$ and there is $u\leq s,t$ such that $(u,\varphi)\in \Omega$. The equivalence class of $(s,\varphi)$ will be called the {\em germ of} $s$ {\em at} $\varphi$ and will be denoted by $[s,\varphi]$.

\begin{lemma} If $(s,\varphi)\sim (t,\varphi)$ then $s\circ \varphi = t\circ \varphi$.
\end{lemma}
\begin{proof}
By assumption, there is a $u\leq s,t$ such that $\varphi\in D_{u^*}$. It suffices to show that $s\circ \varphi = u\circ \varphi$. Since $u\leq s$, we have $u=su^*$. Then for any $e\in P(S)$ we have 
$(u\circ \varphi)(e) = \varphi((eu)^*) = \varphi((esu^*)^*) = \varphi ((es)^*u^*) = \varphi ((es)^*)\varphi(u^*) = \varphi((es)^*) = (s\circ \varphi)(e)$, so that $s\circ \varphi = u\circ \varphi$, as required.
\end{proof}

Let $C_1 = \Omega/\sim$ be the set of germs and $C_0=\widehat{P(S)}$. We define the maps $d,r\colon C_1\to C_0$ and $u\colon C_0\to C_1$ by
\begin{equation}\label{eq:oper1}
d([s,\varphi]) = \varphi, \,\, r([s,\varphi]) = s\circ \varphi, \,\, u(\varphi) = [e,\varphi],
\end{equation}
where $e\in P(S)$ is such that $\varphi(e)=1$ (it is easy to see that $u(\varphi)$ is well defined).
We further set 
$$
C_2 = \{([s,\varphi], [t,\psi]) \in C_1\times C_1
\colon \varphi = t\circ \psi\}$$
and for $([s,\varphi], [t,\psi]) \in C_2$  define the map $m\colon C_2\to C_1$, $([s,\varphi], [t, \psi]) \mapsto [s,\varphi] \cdot [t, \psi]$ by
\begin{equation}\label{eq:oper2}
[s,\varphi] \cdot [t, \psi] = [st, \psi].
\end{equation}

\begin{lemma} The product $[s,\varphi] \cdot [t, \psi]$ is well defined.
\end{lemma}

\begin{proof} We have $\psi(t^*) = \varphi(s^*)=1$, so $\psi((st)^*) = \psi(s^*t)^* = (t\circ \psi)(s^*) = \varphi(s^*)=1$. Furthermore, if $(s,\varphi) \sim (p,\varphi)$ and $(t,\psi) \sim (q,\psi)$, one easily checks that
$(st, \psi) \sim (pq,\psi)$.
\end{proof}

\begin{lemma}
$C = (C_1, C_0, d,r,u,m)$ is a (small) category. 
\end{lemma}

\begin{proof} The proof amounts to verifying the axioms of a small category, given in Subsection \ref{subs:discr_cat}. We check the identity $1_{r(x)}x=x$ which is a part of axiom (UL). Let $x=[s,\varphi]\in C_1$. Then $1_{r(x)} = ur(x) = [e, s\circ \varphi]$ where $(s\circ \varphi)(e)=1$. Then $1_{r(x)}x = [e, s\circ\varphi][s,\varphi] = [es,\varphi]$. Since $es\leq s$ and $\varphi((es)^*) = 1$, we have $[es,\varphi] = [s,\varphi]$, so that $1_{r(x)}x=x$. The remaining axioms can be checked similarly (or easier).
\end{proof}

We now topologize the category $C$. We topologize $C_0=\widehat{P(S)}$ with the locally compact Stone space topology induced by the Boolean algebra $P(S)$. Recall that the compact-open sets $D_e$, where $e$ runs through $P(S)$, form a basis of this topology. To topologize $C_1$,  for each $s\in S$ we set
$\Theta(s)$ to be the set of all the germs $[s,\varphi]\in C_1$ whose first component is~$s$. 

\begin{lemma}\label{lem:aux1}\mbox{}
\begin{enumerate}
\item If $s\leq t$ then $\Theta(s)\subseteq \Theta(t)$,
\item $d(\Theta(s)) = D_{s^*}$.
\end{enumerate}   
\end{lemma}

\begin{proof} (1) If $[s,\varphi]\in \Theta(s)$ and $s\leq t$ then $1\geq \varphi(t^*)\geq \varphi(s^*) =1$, so $[s,\varphi] = [t,\varphi]$ and $[s,\varphi]\in \Theta(t)$. 

(2) By the definition, $d(\Theta(s)) = \{\varphi \colon [s,\varphi]\in \Theta(s)\} = \{\varphi\colon \varphi(s^*)=1\} = D_{s^*}$.
\end{proof}

\begin{lemma}
The collection of all the sets $\Theta(s)$, where $s$ runs through $S$, forms the basis of a topology on $C_1$. 
\end{lemma}

\begin{proof}
Let $s,t\in S$ and $[u,\psi] \in \Theta(s)\cap \Theta(t)$. This means that $[u,\psi] = [s,\psi] = [t,\psi]$. Hence there is $v\leq u,s,t$ such that  
$[u,\psi] = [v,\psi]$. Then $\Theta(v)\subseteq \Theta(s)\cap \Theta(t)$ by Lemma \ref{lem:aux1}(1) and $[u,\psi] = [v,\psi]\in \Theta(v)$. The statement follows.
\end{proof}

\begin{lemma}
 $C = (C_1, C_0, d,r,u,m)$ is a topological category. \end{lemma}

\begin{proof} We need to prove that the structure maps $d,r,u$ and $m$ are continuous. Let $D_e$ be a basic open set of $C_0$ (see Section \ref{s:stone_gba}). Since 
$$[s,\varphi]\in d^{-1}(D_e) \Leftrightarrow \varphi(e) = 1 \Leftrightarrow \varphi(s^*e)=1 \Leftrightarrow [s,\varphi] = [se, \varphi] \Leftrightarrow [s,\varphi]\in \Theta(se),$$
we have $d^{-1}(D_e) = \cup \{\Theta(se) \colon s\in S\} = \cup \{\Theta(t) \colon t^*\leq e\}$, which is an open set in $C_1$. Thus $d$ is continuous. Since
\begin{multline*}
[s,\varphi] \in r^{-1}(D_e) \Leftrightarrow s\circ \varphi \in D_e \Leftrightarrow (s\circ \varphi)(e) = 1 \Leftrightarrow \varphi((es)^*) = 1 \\
\Leftrightarrow [s,\varphi] = [es,\varphi] \Leftrightarrow [s,\varphi] \in \Theta(es),
\end{multline*}
we have $r^{-1}(D_e) = \cup\{\Theta(s) \colon es=s\}$, which is an open set in $C_1$, so $r$ is continuous. 
If $s\in S$, then $u^{-1}(\Theta(s)) = \{\varphi\in C_0\colon u(\varphi)=[e,\varphi] \in \Theta(s)\} = \{\varphi\in C_0\colon [e,\varphi] = [s,\varphi]\}$. This is the set of those $\varphi\in C_0$, for which there is $e\leq s$ with $\varphi(e)=1$. Being the union of all $D_e$ where $e\leq s$ is a projection, it is an open set.

We finally show that $m$ is continuous. Let $W$ be open in $C_1$, we show that $m^{-1}(W)$ is open in $C_2$. Let $([s,t\circ \psi], [t,\psi])\in m^{-1}(W)$. Since the sets $\Theta(s)$ form a basis of the topology on $C_1$, there is some $r\in S$ such that 
$[s,t\circ \psi][t,\psi] = [st,\psi] \in \Theta(r)\subseteq W$.
Hence $[st,\psi] = [r,\psi]$, so there is $e\in P(S)$, such that $ste = re$ and $\psi((ste)^*) =1$. Since $1=\psi((ste)^*) = \psi((st)^*)\psi(e)$, it follows that $\psi(e)=1$. Then $[st,\psi] = [r,\psi] = [re,\psi]$. The set $A=(\Theta(s)\times \Theta(te))\cap C_2$ is open in $C_2$ and contains $([s,te\circ \psi],[te,\psi])$. Moreover, for all $([s,te\circ \varphi][te,\varphi]) \in A$ we have $[ste, \varphi] = [re, \varphi]\in \Theta(re) \subseteq \Theta(r)$, which completes the proof. \end{proof}

From now on we consider  $C = (C_1, C_0, d,r,u,m)$ as a topological category with the defined topology.

\begin{lemma}\label{lem:theta_s}
Let $s\in S$. Then $\Theta(s)$ is a compact slice. 
\end{lemma}

\begin{proof}
Since $d([s,\varphi]) = \varphi$, it is immediate that $d|_{\Theta(s)}\colon \Theta(s) \to d(\Theta(s))$ is a  bijection. Since $\Theta(s)$ is open, $d|_{\Theta(s)}$ is continuous. In addition, if $\Theta(t)\subseteq \Theta(s)$ then $d(\Theta(t)) = D_{t^*}$ is an open subset of $D_{s^*}$, so that the map $d|_{\Theta(s)}$ is open. Hence it is a homeomorphism.
\end{proof}

\begin{corollary}\label{cor:ample_cat}
 $C = (C_1, C_0, d,r,u,m)$ is an ample category.
\end{corollary}

\subsection{The cases of a  range semigroup and of an \'etale range semigroup}
\begin{proposition}\label{prop:open1}
Suppose that $S$ is a preBoolean range semigroup. Then the map $r$ of the category $C = (C_1, C_0, d,r,u,m)$ is open and $D_{s^+}= r(\Theta(s))$.  \end{proposition}

\begin{proof}
We show that $r(\Theta(s)) = D_{s^+}$. Let $\varphi\in D_{s^*}$ and $e\in P(S)$. Since $(s\circ\varphi)(s^+) = \varphi((s^+s)^*) = 1$, we have $s\circ\varphi \in D_{s^+}$. Hence $r(\Theta(s))\subseteq D_{s^+}$.

For the reverse inclusion, let $\psi \in D_{s^+}$ and define 
$F=\{(es)^*\colon \psi(e)=1\}^{\uparrow}$ (where  $A^{\uparrow} = \{b\colon b\geq a \text{ for some }a\in A\}$ denotes the upward closure of $A$). If $x,y\in F$ with $x\geq (es)^*$, $y\geq (fs)^*$ with $\psi(e)=\psi(f) = 1$, we have
that $xy\geq (es)^*(fs)^* = (es(fs)^*)^*=(efs)^*$ and $\psi(ef) = \psi(e)\psi(f) = 1$, so that $xy\in F$. It follows that $F$ is a filter. Let $A\supseteq F$ be a prime filter and let $\varphi\colon P(S)\to {\mathbb B}$ be the prime character with $A=\varphi^{-1}(1)$. We show that $s\circ \varphi = \psi$. Let $e\in P(S)$ and show that
$\varphi((es)^*) = \psi(e)$. If $\psi(e)=1$ then $(es)^*\in F\subseteq A$, so that $\varphi((es)^*)=1$. Suppose that $\psi(e) = 0$ but $\varphi((es)^*)=1$. Let $f\in P(S)$ be such that $ef=0$ and $\psi(f) = 1$. Then $(fs)^* \in F$ so that $\varphi((fs)^*)=1$. But this yields
$1 = \varphi((es)^*)\varphi((fs)^*) = \varphi((es)^*(fs)^*) = \varphi((efs)^*) =0$, which is a contradiction. It follows that if $\psi(e) = 0$ then necessarily also $\varphi((es)^*)=0$, and the equality $s\circ \varphi = \psi$ is shown. Hence $\psi = r([s, \varphi])$, so that $D_{s^+}\subseteq r(\Theta(s))$.

It follows that $D_{s^+}= r(\Theta(s))$. Since  the sets $\Theta(s)$ form a basis of the topology on $C_1$ and $D_{s^+}$ is open, the map  $r$ is open.
\end{proof}

\begin{lemma}\label{lem:8j1}
Let $S$ be a preBoolean range semigroup and $s\in S$ be bideterministic. Then $\Theta(s)$ is a bislice.    
\end{lemma}

\begin{proof}  We need only to check that $\Theta(s)$ is a local cosection. Suppose $[s,\varphi], [s,\psi]\in \Theta(s)$ are different and show that $r([s,\varphi])\neq r([s,\psi])$. Since $\varphi\neq \psi$, there is $e\in P(S)$ such that $\varphi(e)=1$ and $\psi(e)=0$. Since $s$ is bideterministic, there is $f\in P(S)$ such that $se = fs$. Then $r([s,\varphi])(f) = (s\circ \varphi)(f) = \varphi((fs)^*) = \varphi((se)^*) = \varphi(s^*e) = 1$ and similarly $r([s,\psi])(f) = \psi(s^*e) = 0$.
Hence $r([s,\varphi])\neq r([s,\psi])$, as needed.
\end{proof}

We now turn to \'etale preBoolean range semigroups.

\begin{proposition} \label{prop:lh1}
Suppose that $S$ is an \'etale preBoolean range semigroup. 
Then the map $r$ of the category $C = (C_1, C_0, d,r,u,m)$ is a local homeomorphism. Furthermore, the category $C$ coincides with the category of germs of the preBoolean birestriction semigroup ${\mathscr{BD}}(S)$.
\end{proposition}

\begin{proof}
Let $[s,\varphi]\in C_1$. We first show that $s$ can be chosen bideterministic. Since $S$ is \'etale, there are bideterministic elements $s_1,\dots, s_n$ where $n\geq 1$ such that $s= s_1\vee \cdots \vee s_n$. By Lemma \ref{lem:joins} we have
$s^* = s_1^* \vee \cdots \vee s_n^*$. Since $\varphi(s^*)=1$, there is $i$ such that $\varphi(s_i^*) =1$, so that $[s,\varphi] = [s_i, \varphi]$, as needed. It follows that $C$ is the category of germs of ${\mathscr{BD}}(S)$.

The map $r$ is continuous by assumption and open by Proposition \ref{prop:open1}. Furthermore, it is surjective since for $\varphi\in \widehat{P(S)}$ there is an $e$ with $\varphi(e)\neq 0$, so that the germ $[e,\varphi]$ is well defined and $r([e,\varphi]) = \varphi$. 
Let $[s,\varphi]\in C_1$ where $s$ is bideterministic. Then $\Theta(s)$ is a neighborhood of $[s,\varphi]$ and is a bislice by Lemma \ref{lem:8j1}. It follows that $r|_{\Theta(s)}$ is injective, and so is a homeomorphism.
\end{proof}

\subsection{Morphisms}\label{subs:morphisms2}
Let $S,T$ be preBoolean restriction semigroups with local units and let $f\colon S\to T$ be a morphism.
Let $C^S = (C_1^S, C_0^S)$ and $C^T = (C_1^T, C_0^T)$ be their respective categories of germs. 

Since $f|_{P(S)}\colon P(S) \to P(T)$ is a proper morphism of generalized Boolean algebras, Theorem \ref{th:Stone_classical} implies that the map $(f|_{P(S)})^{-1}\colon C_0^T \to C_0^S$, $\varphi \mapsto \varphi f|_{{P(S)}}$ is a proper and continuous map. We put $\bar{f}(\varphi) =  \varphi f|_{P(S)}$. We  construct a cofunctor $(\mu, \bar{f}, \rho)$ from $C^S$ to $C^T$. We first define the action $(\mu, \bar{f})$ of $C^S$ on $C_0^T$.
Since $\bar{f}(\psi) = \psi f|_{P(S)} = d([s, \psi f|_{P(S)}]))$, the elements of $C_1^S \times_{\bar{f},d} C_0^T$ are of the form $([s,\psi f|_{P(S)}], \psi)$. Then $1= \psi f(s^*) = \psi(f(s)^*)$, so that $[f(s), \psi] \in C_1^T$. We put $\mu([s,\psi f|_{P(S)}], \psi) = r([f(s), \psi]) = f(s)\circ \psi$.  We now define the functor $\rho\colon C^S\rtimes C_0^T \to C^T$ by $\rho_0(x)=x$ for all $x\in C_0^T$ and $\rho_1([s,\psi f|_{P(S)}], \psi) = [f(s),\psi]$. It is routine to verify that $(\mu, \bar{f},\rho)$ is a cofunctor.

\begin{lemma} \label{lem:9ja}
Let $S,T$ be \'etale preBoolean range semigroups and $f\colon S\to T$ a morphism which preserves bideterministic elements. Then the action $(\mu, \bar{f})$ is injective.   \end{lemma}

\begin{proof}
From the converse, suppose $(\mu, \bar{f})$ is not injective. Then there are distinct $\varphi,\psi\in C_0^T$ with $\varphi f|_{P(S)} = \psi f|_{P(S)}$ and $s\in S$ with $\varphi f(s^*)=1$ such that $\mu([s,\psi f|_{P(S)}],\psi) = \mu([s,\varphi f|_{P(S)}],\varphi)$. Since $S$ is \'etale, we have that $s=s_1\vee \cdots\vee s_n$ with all $s_i$ bideterministic, which implies that $[s,\varphi f|_{P(S)}] = [s_i, \varphi f|_{P(S)}]$ for some $i$. So we can assume that $s$ is bideterministic. The assumption implies that $f(s)\circ \varphi = f(s)\circ \psi$.  We have that $[f(s),\varphi], [f(s),\psi]\in \Theta(f(s))$ are distinct but 
$r([f(s),\varphi]) = r([f(s),\psi])$, so that $\Theta(f(s))$ is not a bislice. By Lemma \ref{lem:8j1} then $f(s)$ is not bideterministic, which is a contradiction. 
\end{proof}

\begin{remark} \label{rem:r1} Let $S,T$ be Boolean birestriction semigroups and $f\colon S\to T$ morphism between them. Since $S$ and $T$ are \'etale preBoolean range semigroups and all elements are bideterministic, Lemma \ref{lem:9ja} implies that the  cofunctor $(\mu, \bar{f}, \rho)$ is injective.
\end{remark}

\section{The main results}\label{s:dualities}
\subsection{The adjunction and the equivalences} \label{subs:cat1}
We define morphisms of type $i=1,2,3,4$ between preBoolean restriction semigroups with local units as follows:

\begin{center}
\begin{tabular}{|c|l|} 
 \hline
 type 1 & morphisms \\ 
 \hline
 type 2 & weakly meet-preserving morphisms \\ 
 \hline
 type 3 & proper morphisms \\ 
 \hline
  type 4 & proper and weakly meet-preserving morphisms \\
 \hline
\end{tabular}
\end{center}

We also define morphisms of type $i$ between ample topological categories:

\begin{center}
\begin{tabular}{|c|l|} 
 \hline
 type 1 & cofunctors \\ 
 \hline
 type 2 & cofunctors which are injective on arrows \\ 
 \hline
 type 3 & cofunctors which are surjective on arrows \\ 
 \hline
  type 4 & cofunctors which are bijective on arrows \\
 \hline
\end{tabular}
\end{center}

We aim to show that the passages from a Boolean restriction semigroup with local units to its category of germs and from an ample category to its Boolean restriction semigroup of compact slices are mutually inverse. However, we start from a more general adjunction result with preBoolean restriction semigroups on the algebraic side. This adjunction will yield an equivalence not only for Boolean restriction semigroups with local units but also for Boolean  birestriction semigroups on the algebraic side.

Let ${\mathsf{PBR}}_i$  denote the category of preBoolean restriction semigroups with local units and morphisms of type $i$
and ${\mathsf{AC}}_i$ the category of ample topological categories and morphisms of type $i$. The constructions in Section \ref{s:cat_to_sem} define the functors
\begin{equation}\label{eq:functor1}
{\mathsf{S}}_i\colon {\mathsf{AC}}_i \to {\mathsf{PBR}}_i
\end{equation}
which assign to an ample category $C$ the Boolean restriction semigroup $C^a$ and to a morphism $F\colon C\rightsquigarrow D$ the morphism $F_*\colon C^a \to D^a$.
The constructions of Section \ref{s:sem_to_cat} define the functors \begin{equation}\label{eq:functor1a}
{\mathsf C}_i\colon {\mathsf{PBR}}_i \to \mathsf{AC}_i
\end{equation}
which assign to a preBoolean restriction semigroup with local units  its category of germs and to a morphism $S\to T$ the cofunctor ${\mathsf C}(S)\rightsquigarrow {\mathsf C}(T)$ constructed in subsection \ref{subs:morphisms2}.

In the category ${\mathsf{Top}}_1$ of topological categories and functors between them, an isomorphism from $C$ to $D$ is a functor $f=(f_1,f_0)\colon C\to D$ where $f_1\colon C_1\to D_1$ and $f_0\colon C_0\to D_0$ are homeomorphisms which commute with the structure maps. We now define the appropriate notion of an isomorphism in the category ${\mathsf{Top}}_2$ of topological categories and cofunctors between them.

Let $f=(f_1,f_0)\colon C\to D$ be an isomorphism in the category ${\mathsf{Top}}_1$. Then the category $C\rtimes D_0$ is isomorphic to $C$ and we have the action $(\mu, f_0^{-1})$ of $C\rtimes D_0$ on $D_0$ given by $\mu(s, x) = f_0(s\circ f_0^{-1}(x))$.
There is an isomorphism $\rho\colon C\rtimes D_0\to D$ given on arrows by $(s,x)\mapsto f_1(s)$ and which is identical on objects. This suggests to call a cofunctor $F=(\mu, f, \rho)\colon C\rightsquigarrow D$ an {\em isomorphism} if $f\colon D_0\to C_0$ is a homeomorphism, the action of $C\rtimes D_0$ (which is isomorphic to $C$) on $C_0$ is equivalent, via $f$, to the natural action of $C$ on $C_0$, and the functor $\rho \colon C\rtimes D_0 \to D$ is an isomorphism. The definition implies that the categories $C$ and $D$ are isomorphic in the category ${\mathsf{Top}}_1$ if and only if they are isomorphic in ${\mathsf{Top}}_2$.

\begin{theorem} (Adjunction theorem) \label{th:adjunction1}\mbox{}
\begin{enumerate}
\item
For each $i=1,2,3,4$ the functors
$$
{\mathsf{C}}_i\colon {\mathsf{PBR}}_i \mathrel{\mathop{\rightleftarrows}} {\mathsf{AC}}_i \colon {\mathsf{S}}_i
$$
establish an adjunction ${\mathsf{C}}_i \dashv {\mathsf{S}}_i$ between the categories ${\mathsf{PBR}}_i$ and ${\mathsf{AC}}_i$.
The component at an object $S$ of the unit  $\eta\colon 1_{{\mathsf{PBR}}_i} \to {\mathsf{S}}_i{\mathsf{C}}_i$ is the morphism 
$$\eta_S\colon S\to {\mathsf S}_i{\mathsf C}_i(S)$$
given by 
$$s\mapsto \Theta(s).
$$
The component at an object $C$ of the counit  $\varepsilon\colon {\mathsf{C}}_i{\mathsf{S}}_i \to 1_{{\mathsf{AC}}_i}$ is the cofunctor $$\varepsilon_C = (\mu, f, \rho) \colon {\mathsf C}_i{\mathsf S}_i(C)\rightsquigarrow C$$ given by 
$$
\mu([A,\varphi_x],x) = r(A1_x), \,\, f(x) = \varphi_x,  \,\, \rho_1([A,\varphi_x],x) = A1_x, \,\, \rho_0(x)=x,
$$
where $\varphi_x$ is as in Theorem \ref{th:Stone_classical}. Moreover, $\varepsilon_C$ is an isomoprhism for every object $C$ of the category ${\mathsf{AC}}_i$.

\item The adjunction ${\mathsf{C}}_i \dashv {\mathsf{S}}_i$ restricts (for each type of morphisms) to an adjunction  between the category of preBoolean range semigroups  and the category of strongly ample categories, and furthermore to an adjunction between the category of \'etale preBoolean range semigroups and the category of biample categories.
\end{enumerate}
\end{theorem}

\begin{proof}
(1) We first show that the map $\eta_S$  is a $(2,1)$-morphism. Let $s,t\in S$. Since elements of $\Theta(s)$ and $\Theta(t)$ are of the form $[s,\varphi]$ and $[t, \psi]$, respectively, their products are of the form $[st, \psi]$, and we have the inclusion $\Theta(s)\Theta(t) \subseteq \Theta(st)$. Let $[st,\psi] \in \Theta(st)$. Since $(t\circ \psi)(s^*) = \psi((s^*t)^*) = \psi((st)^*) = 1$ and $1\geq \psi(t^*) \geq \psi ((st)^*)=1$, the germs $[s, t \circ \psi]$ and $[t,\psi]$ are well defined, so that $[st, \psi] = [s,t\circ \psi][t,\psi]\in \Theta(s)\Theta(t)$, and so we have shown the equality $\Theta(s)\Theta(t) = \Theta(st)$.

Furthermore, for $s\in S$, we have $$(\Theta(s))^* = 1_{d(\Theta(s))} = \{1_{\varphi}\colon \varphi(s^*)=1\} = \{[s^*,\varphi]\colon \varphi(s^*)=1\} = \Theta(s^*).$$

\noindent By Theorem \ref{th:Stone_classical} we have that $\eta_S|_{P(S)}$ is a morphism of generalized Boolean algebras.

Let us show that
if $s\smile t$ and $s\vee t$ exists then   $\Theta(s\vee t) = \Theta(s) \cup \Theta(t)$.
By Lemma \ref{lem:aux1}(1) we have $\Theta(s\vee t) \supseteq \Theta(s), \Theta(t)$, so that $\Theta(s\vee t) \supseteq \Theta(s) \cup \Theta(t)$. For the reverse inclusion it suffices to show that $d(\Theta(s\vee t)) \subseteq d(\Theta(s) \cup \Theta(t)) = D_{s^*}\cup D_{t^*}$. If $\varphi \in d(\Theta(s\vee t))$, we have $\varphi ((s\vee t)^*)=1$. By Lemma \ref{lem:joins} we obtain $\varphi(s^*)\vee \varphi(t^*) = 1$, so that either $\varphi\in D_{s^*}$ or $\varphi \in D_{t^*}$, as desired. Hence $\eta_S$ is a morphism of the category ${\mathsf{PBR}}_i$. The verification of naturality of $\eta_S$ is routine and we omit it.

We now prove that $\varepsilon_C$ is a cofunctor and is an isomorphism. We first check that 
\begin{equation}\label{eq:s28a}
A\circ \varphi_x = \varphi_{r(A1_x)},
\end{equation}
where $A$ is a compact slice of $C$. We have $\varphi_{r(A1_x)}(E) = 1$ if and only if $1_{r(A1_x)}\in 1_E$. On the other hand, $(A\circ \varphi_x)(1_{E}) = 1$ if and only if $\varphi_x((1_EA)^*) = 1$ which is equivalent to $1_x\in (1_EA)^*$. Since $A$ is a slice, this is equivalent to $A1_x \in A(1_EA)^* = 1_EA$, which holds if and only if $1_{r(A1_x)}\in 1_E$, as needed. 
To show that $(\mu,f)$ is an action, we verify that (A1), (A2) and (A3) hold. That axioms (A1) and (A2) hold quickly follows from \eqref{eq:s28a}. For (A3), we need to check that $1_{\varphi_x}\cdot x = x$. Recall that $1_{\varphi_x} = [1_U,\varphi_x]$ where $U$ is such that $\varphi_x(U)=1$. So we have $1_{\varphi_x}\cdot x = [1_U,\varphi_x]\cdot x = r(1_U1_x) = r(1_x)= x$ since $x\in U$. The verification that $\mu$ is continuous is left to the reader. 

It follows from \eqref{eq:s28a} that the action of  ${\mathsf C}_i{\mathsf S}_i(C) \rtimes C_0$ on $C_0$ given by $[A,\varphi_x]\cdot x = r(A1_x)$ is equivalent, via $f\colon x\mapsto \varphi_x$, to the natural action of ${\mathsf C}_i{\mathsf S}_i(C)$ on its space of identities given  by $[A,\varphi_x]\cdot \varphi_x = r([A,\varphi_x])=\varphi_{r(A1_x)}$.
We verify that  $\rho\colon {\mathsf C}_i{\mathsf S}_i(C) \rtimes C_0\to C$ is an isomorphism. That $\rho$ is functorial quickly follows from the fact that, in the category $ {\mathsf C}_i{\mathsf S}_i(C) \rtimes C_0$, we have $r([A,\varphi_x],x) = r(A1_x)$. Since $\rho_0\colon C_0\to C_0$ is the identity map, to show that $\rho$ is an isomorphism, we only show that $\rho_1$ is a homeomorphism. We prove that it is a bijection.
Suppose $\rho_1([A,\varphi_x],x) = \rho_1([B,\varphi_y],y)$, that is, $A1_x = B1_y$. Then $x=d(A1_x) = d(B1_y) = y$. So we have $A1_x = B1_x$. It follows that $A\cap B$ is a non-empty set and contains $A1_x$. Since compact slices form a basis of the topology on $C_1$, there is a compact slice $U$ which contains $1_x$ and is contained in both $A$ and $B$. It follows that $[A,\varphi_x] = [U,\varphi_x] = [B,\varphi_x]$, and we have proved that $\rho_1$ is injective. Let us show it is surjective. Let $u\in C$. Since compact slices form a basis of the topology, there is a compact slice, $U$, which contains $u$. But then $u = U1_{d(u)} = \rho_1([U, \varphi_{d(u)}], d(u))$, as needed. 
We now prove that $\rho_1$ is continuous.
It suffices to show that $\rho_1^{-1}(B)$ is open for a compact slice $B$ of $C_1$. It is clear that $\{([B, \varphi_{d(x)}],d(x))\colon x\in B\}\subseteq \rho_1^{-1}(B)$. Suppose that $[A, \varphi_{d(y)}] \in \rho_1^{-1}(B)$. Then $y\in B$, so $y\in A\cap B$ and there is a compact slice $U\subseteq A\cap B$ containing $y$. It follows that $[A,\varphi_{d(y)}] = [B,\varphi_{d(y)}]$ and shows that $\rho_1^{-1}(B) = \{([B, \varphi_{d(x)}],d(x))\colon x\in B\}$ which is homeomorphic to the compact slice $\Theta(B)$. Since the image under $\rho_1$ of the set $\{([B, \varphi_{d(x)}],d(x))\colon x\in B\}$ is $B$, the map $\rho_1$ is open. We have thus proved that $\rho$ it is an isomorphism, as desired.  The verification of naturality of $\varepsilon_C$ is routine and we omit it. 

We verify that the composite natural transformation
$$
{\mathsf{S}}_i \stackrel{\eta {\mathsf{S}}_i}{\xrightarrow{\hspace{1.2cm}}} {\mathsf{S}}_i{\mathsf{C}}_i{\mathsf{S}}_i \stackrel{{\mathsf{S}}_i\varepsilon}{\xrightarrow{\hspace{1.2cm}}} {\mathsf{S}}_i$$
is the identitity.  Let $C$ be an object of ${\mathsf{AC}}_i$. The morphism $\eta_{{\mathsf{S}}_i(C)}$ takes $s\in {\mathsf{S}}_i(C)$ to $\Theta(s) \in  {\mathsf{S}}_i{\mathsf{C}}_i{\mathsf{S}}_i(C)$ and the morphism ${\mathsf{S}}_i(\varepsilon_C)$ takes $\Theta(s)$ back to $s$, as needed.
We verify that the composite natural transformation
$${\mathsf{C}}_i \stackrel{{\mathsf{C}}_i\eta}{\xrightarrow{\hspace{1.2cm}}} {\mathsf{C}}_i{\mathsf{S}}_i{\mathsf{C}}_i \stackrel{\varepsilon{\mathsf{C}}_i}{\xrightarrow{\hspace{1.2cm}}} {\mathsf{C}}_i
$$
is the identity. Let $S$ be an object of ${\mathsf{PBR}}_i$. We denote ${\mathsf{C}}_i(\eta_S)$ by $F = (\bar{\mu},\bar{f},\bar{\rho})$ and $\varepsilon_{{\mathsf{C}}_i(S)}$ by $G = (\tilde{\mu},\tilde{f},\tilde{\rho})$. We have $\bar{f}(\varphi_{\psi}) = \psi$, $\bar{\mu}([s,\psi], \varphi_{\psi}) = r([\Theta(s), \varphi_{\psi}])$ and $\bar{\rho}_1([s,\psi], \varphi_{\psi}) = [\Theta(s), \varphi_{\psi}]$. Furthermore,  $\tilde{f}(\psi) = \varphi_{\psi}$, $\tilde{\mu}([\Theta(s), \varphi_{\psi}], \psi) = r([s,\psi])$ and $\tilde{\rho}_1([\Theta(s), \varphi_{\psi}], \psi) = [s,\psi]$. It follows that the composite cofunctor 
$G\circ F = (\mu, f, \rho)$ satisfies $f(\psi)=\psi$, $\mu([s,\psi],\psi) = r([s,\psi])$ 
and $\rho\colon {\mathsf C}_i(S)\to {\mathsf C}_i(S)$ is the identity functor. Hence $G\circ F$ is the identity on ${\mathsf C}_i(S)$. Applying \cite[Chapter~IV, Theorem~2(v)]{M71} we conclude that $ {\mathsf{C}}_i \dashv  {\mathsf{S}}_i$ is an adjunction with unit $\eta$ and counit $\varepsilon$.

(2) This part follows from Propositions \ref{prop:range} and \ref{prop:etale}, and from Propositions \ref{prop:open1} and \ref{prop:lh1}, respectively. 
\end{proof}

\begin{proposition}\label{prop:epsilon} Let $\eta$ be the unit of the adjunction ${\mathsf C}_i \dashv {\mathsf S}_i$ of Theorem \ref{th:adjunction1}.
Then for any object $S$ of the category ${\mathsf{PBR}}_i$ the morphism $\eta_S$ is injective. Furthermore, $\eta_S$ is an isomorphism if and only if $S$ is Boolean.
\end{proposition}

\begin{proof}
To prove that $\eta_S$ is injective, we assume that $s,t\in S$ are such that $\Theta(s)=\Theta(t)$ and need to show that $s=t$. We have $d(\Theta(s))= d(\Theta(t))$, so that $D_{s^*}= D_{t^*}$ and thus $s^*=t^*$, by Theorem \ref{th:Stone_classical}. Let $\varphi \in D_{s^*}$. Then $[s,\varphi]\in \Theta(s)=\Theta(t)$, so there is $u_{\varphi}\leq s,t$ with $\varphi(u_{\varphi}^*)=1$. We have $D_{s^*} = \cup_{\varphi\in D_{s^*}} D_{u_{\varphi}^*}$. By compactness, there is a finite set $I\subseteq D_{s^*}$ such that $D_{s^*} = \cup_{\varphi\in I} D_{u_{\varphi}^*}$. Hence $s^* = \vee_{\varphi\in I} u_{\varphi}^*$, again by Theorem \ref{th:Stone_classical}. Applying Lemma~\ref{lem:left_distr} and $u_{\varphi}\leq s$ we obtain $s= ss^* = s(\vee_{\varphi\in I} u_{\varphi}^*) = \vee_{\varphi\in I} su_{\varphi}^* = \vee_{\varphi\in I}u_{\varphi}$.
Similarly, we have $t = \vee_{\varphi\in I}u_{\varphi}$. It follows that $s=t$, as desired. 

If $\eta_S$ is an isomorphism then $\eta_S(S)$ is isomorphic to ${\mathsf S}_i{\mathsf C}_i(S)$ which is Boolean. Conversely, suppose $\eta_S(S)$ is Boolean. We show that $\eta_S$ is surjective.
Let $A$ be a compact slice of ${\mathsf C}_i(S)$. Since the elements $\Theta(s)$, where $s$ runs through $S$, form a basis of the topology, $A = \cup_{i\in I} \Theta(s_i)$ for some finite set of indices $I$ and some $s_i\in S$, $i\in I$. Since the slices $\Theta(s_i)$ have an upper bound, $A$, they are  compatible. Then $\Theta(s_i)\Theta(s_j^*) = \Theta(s_j)\Theta(s_i^*)$ for all $i,j$, so that $\Theta(s_is_j^*) = \Theta(s_js_i^*)$, and injectivity of $\Theta$ implies that $s_is_j^* = s_js_i^*$, so that $s_i \smile s_j$ for all $i,j\in I$. Since $S$ is Boolean, the join $\vee_{i\in I}s_i$ exists in $S$. It now follows  that
$A=\cup_{i\in I}\Theta(s_i) = \Theta(\vee_{i\in I}s_i)$, so that $s\mapsto \Theta(s)$ is surjective. So we have proved that $\eta_S$ is an isomorphism if and only if $S$ is Boolean. 
\end{proof}

Since the counit of the adjunction ${\mathsf C}_i \dashv {\mathsf S}_i$ of Theorem \ref{th:adjunction1} is a natural isomorphism, the right adjoint functor ${\mathsf S}_i$ is full and faithful, so the adjunction is in fact a reflection (see \cite[Exercise 2.2.12]{Le14}). It follows from \cite[Corollary 4.2.4]{Bo94}) that the adjunction is monadic and the category of Eilenberg-Moore algebras of the induced monad is equivalent to the category of Boolean restriction semigroups with local units and morphisms of type $i$ which is a reflective subcategory of ${\mathsf{PBR}}_i$.

The fixed-point equivalences of the adjunctions of Theorem \ref{th:adjunction1}
are the following equivalences of categories.

\begin{theorem} (Equivalence theorem 1) \label{th:eq1} 
\begin{enumerate}
    \item The category of Boolean restriction semigroups with local units and morphisms of type $i$ is equivalent to the category of ample categories and morphisms of type $i$.
    \item The category of Boolean range semigroups and morphisms of type $i$ is equivalent to the category of strongly ample categories and morphisms of type $i$.
    \item The category of \'etale Boolean range semigroups and morphisms of type $i$ is equivalent to the category of biample categories and morphisms of type $i$.
\end{enumerate}
\end{theorem}

\begin{remark} Note that morphisms $S\to T$ on the algebraic side of our equivalences are presented by cofunctors ${\mathsf{C}}_i(S)\rightsquigarrow {\mathsf{C}}_i(T)$, whose anchor map points from ${\mathsf{C}}_i(T)_0$  to ${\mathsf{C}}_i(S)_0$. So our covariant equivalences do extend the contravariant Stone duality of Theorem \ref{th:Stone_classical}.  Furthermore, in the case of morphisms of type $4$ these cofunctors can be equivalently described as certain functors ${\mathsf{C}}_i(T) \to {\mathsf{C}}_i(S)$, and we obtain genuine dualities, see Subsection~\ref{subs:functors}.
\end{remark}

The following result is a consequence of Proposition \ref{prop:range} and Proposition \ref{prop:open1}. 

\begin{corollary}
Let $S$ be a Boolean restriction semigroup with local units.
Then $S$ admits a structure of a range semigroup if and only if its category of germs is strongly ample. 
\end{corollary}

Let ${\mathsf{BBR}}_{i}$ be the category of Boolean birestriction semigroups and morphisms of type $i$ and ${\mathsf{BAC}}_{i}$ the category of biample categories  and morphisms of type $i$ whose action is injective. The functor ${\mathsf{C}}_i$ from the category of \'etale preBoolean range semigroups and morphisms of type $i$ which preserve bideterministic elements to the category of biample categories with morphisms of type $i$ whose action is injective restricts to the category ${\mathsf{BBR}}_{i}$, and we denote this restriction by the same symbol ${\mathsf{C}}_i$.
Furthermore, let ${\mathsf{BD}}$ be the functor from the category of \'etale Boolean range semigroups and morphisms of type $i$ which preserve bideterministic elements to the category ${\mathsf{BBR}}_{i}$  which maps $S$ to ${\mathscr{BD}}(S)$ and $f\colon S\to T$ to the restriction of $f$ to ${\mathscr{BD}}(S)$.

\begin{theorem} (Equivalence theorem 2)\label{th:eq2} For each $i=1,2,3,4$
the functors 
$$
{\mathsf{C}}_i\colon {\mathsf{BBR}}_i \mathrel{\mathop{\rightleftarrows}} {\mathsf{BAC}}_{i} \colon {\mathsf{BD}}\circ {\mathsf{S}}_i.
$$
establish an equivalence of categories.
\end{theorem}

\begin{proof}
By Theorem \ref{th:adjunction1}  $\eta_S \colon S\to {\mathsf{S}}_i{\mathsf{C}}_i(S)$ is an injective morphism and Lemma \ref{lem:8j1} ensures that $\eta_S(S)$ is contained in ${\mathscr{BD}}({\mathsf{S}}_i{\mathsf{C}}_i(S))$.
If $A\in{\mathscr{BD}}({\mathsf{S}}_i{\mathsf{C}}_i(S))$, then $A$ is a compact bislice, but we know that compact bislices $\Theta(s)$, where $s$ runs through $S$, form a basis of the topology. It follows that $A = \Theta(s_1) \cup \cdots \cup \Theta(s_n)$ for some $s_1,\dots. s_n\in S$, and because $\eta_S(S)$ is closed with respect to finite bicompatible joins, we conclude that $A \in \eta_S(S)$, so that $\eta_S$ is an isomorphism. Since by Proposition \ref{prop:lh1} the category of germs of ${\mathsf S}_i(C)$ coincides with the category of germs of ${\mathscr{BD}}({\mathsf S}_i(C))$, it follows that $\varepsilon_C$ of Proposition \ref{prop:epsilon} is an isomorphism between $C$ and the category of germs of $({\mathsf{BD}}\circ {\mathsf{S}}_i)(C)$. The result follows.
\end{proof}

\begin{remark} That the categories ${\mathsf{BBR}}_i$ and ${\mathsf{BAC}}_{i}$ are equivalent was proved in \cite[Theorem 8.19]{KL17}. Theorem \ref{th:eq1} provides a new proof of the latter equivalence, via the adjunction of Theorem~\ref{th:adjunction1}.
\end{remark}

Combining Theorem \ref{th:eq1} and Theorem \ref{th:eq2}, we obtain the following.

\begin{theorem} (Equivalence Theorem 3) \label{th:eq3}
For each $i=1,2,3,4$ the category of \'etale Boolean range semigroups and morphisms of type $i$ which preserve bideterministic elements is equivalent to the category ${\mathsf{BBR}}_i$ of Boolean birestriction semigroups. 
\end{theorem}

\subsection{The case of meet-semigroups} We say a restriction semigroup $S$ has {\em binary meets} if for every $s,t\in S$ their meet $s\wedge t$ exists in $S$, that is, if $S$ is a lower semilattice with respect to the natural partial order $\leq$\footnote{Meets in restriction categories were introduced and studied in \cite{CGH12a}.}. The following is inspired by
\cite[Proposition 3.7]{St10} and \cite[Corollary 8.3]{K12}.

\begin{proposition} 
Let $S$ be a Boolean restriction semigroup with local units. Then $S$ has binary meets if and only if the space of arrows of its category of germs is Hausdorff.
\end{proposition}

\begin{proof}
let $C=(C_1, C_0)$ be the category of germs of $S$ and $s,t\in S$. Suppose first that $S$ has binary meets.  Let $[s,\varphi], [t,\psi]$ be germs. If $\varphi\neq \psi$, there are disjoint basic neighborhoods $U_{\varphi}$ and $U_{\psi}$ of $\varphi$ and $\psi$ in the Hausdorff space $C_0$. Then $\Theta(s)1_{U_{\varphi}}$ and $\Theta(t)1_{U_{\psi}}$ are disjoint open sets containing $[s,\varphi]$ and $[t,\psi]$, respectively. Suppose now that $\varphi=\psi$. Since $S$ is isomorphic to $\eta_S(S)$, the latter semigroup has binary meets. Then $\Theta(s)\wedge \Theta(t)= \Theta(s)\cap \Theta(t) = \Theta(s\wedge t)$. Since the Boolean algebra of elements below $s$ is isomorphic to the Boolean algebra of the elements below $s^*$ (via the map $u\mapsto u^*$), it follows that $\Theta(s)\setminus \Theta(s\wedge t)$ is a compact slice and a similar conclusion holds for $\Theta(t)\setminus \Theta(s\wedge t)$. Since these two sets are disjoint neighborhoods of $s$ and $t$, respectively, the space $C_1$ is Hausdorff.

Conversely, suppose that $C_1$ is Hausdorff. Since $\Theta(s)$ and $\Theta(t)$ are both compact-open, so is their intersection $\Theta(s)\cap \Theta(t)$. It easily follows that $\Theta(s)\cap \Theta(t)$ is the meet of $\Theta(s)$ and $\Theta(t)$. Since $\eta_S$ is an isomorphism, $S$, too, has binary meets.
\end{proof}

The following can be proved similarly and is known from \cite{KL17}.

\begin{proposition}
Let $S$ be a Boolean birestriction semigroup. Then $S$ has binary meets if and only if the space of arrows of its category of germs is Hausdorff.    
\end{proposition}

We obtain the following corollary.

\begin{corollary} An \'etale Boolean range semigroup $S$ has binary meets if and only if the Boolean birestriction semigroup ${\mathscr{BD}}(S)$ has binary meets.
\end{corollary}

Bearing in mind Remark \ref{rem:meets}, it follows that all our main results admit analogues for meet-semigroups and their morphisms of type $i$, where weakly meet-preserving morphisms (type 2) are replaced by meet-preserving ones. If one restricts attention only to meet-preserving morphisms, then only morphisms of types $2$ and $4$ remain under consideration.

\subsection{The groupoidal case}\label{subs:groupoidal}
The following is well known.

\begin{proposition}\cite{Paterson}\label{prop:pat}
Suppose that $G = (G_1, G_0, u,d,r,m,i)$ is an ample groupoid. Then $\widetilde{G}^a$ is an inverse semigroup.
\end{proposition}

In the next proposition we characterize when a biample category admits a structure of an ample groupoid.

\begin{proposition}\label{prop:6j3}
Suppose $C$ is a biample category. 
The following are equivalent:
\begin{enumerate}
\item There is a homeomorphism $i\colon C_1\to C_1$, so that $(C_1,C_0,u,d,r,m,i)$ is a groupoid.
\item  ${\mathscr{BD}}(C^a) = {\mathscr{I}}(C^a) = \widetilde{C}^a$.
\end{enumerate}
\end{proposition}

\begin{proof}
Suppose that $(C_1,C_0,u,d,r,m,i)$ is a groupoid. Then $\widetilde{C}^a$ is an inverse semigroup by Proposition \ref{prop:pat} and $\widetilde{C}^a = {\mathscr{BD}}(C^a)$ by Proposition \ref{prop:det}(2). Since $\widetilde{C}^a$ is an inverse semigroup, every its element is a partial isomorphism, so that $\widetilde{C}^a \subseteq {\mathscr{I}}(C^a)$. The reverse inclusion holds by Lemma \ref{lem:6j2}.

Suppose now that ${\mathscr{BD}}(C^a) = {\mathscr{I}}(C^a) = \widetilde{C}^a$. Then $\widetilde{C}^a$ is an inverse semigroup, and by the non-commutative Stone duality \cite[Theorems 8.19, 8.20]{KL17} its biample category of germs is a groupoid. 
\end{proof}

Garner \cite{G23b} defines a Boolean restriction monoid $S$ to be {\em \'etale} if everyone of its elements is a join of elements from ${\mathscr{I}}(S)$. Propositions \ref{prop:etale} and \ref{prop:6j3} imply that Garner's \'etale Boolean restriction monoids admit the structure of range  monoids and are a special case of our \'etale Boolean range semigroups arising from ample categories which are groupoids. It is thus natural to call \'etale Boolean range monoids of \cite{G23b} {\em groupoidal}. 

We can readily restrict our results from \'etale Boolean range semigroups to groupoidal \'etale Boolean range semigroups, from biample categories to ample groupoids and from Boolean birestriction semigroups to Boolean inverse semigroups. In particular, the equivalences in Theorem \ref{th:eq2} and Theorem \ref{th:eq3} restrict to the following.

\begin{corollary}\label{cor:last}
The categories of Boolean inverse semigroups, groupoidal \'etale Boolean range semigroups and ample groupoids are pairwise equivalent.   
\end{corollary}

Theorem \ref{th:eq3} implies that an \'etale Boolean range semigroup $S$ can be reconstructed from the Boolean birestriction semigroup ${\mathscr{BD}}(S)$. This is basically because both $S$ and ${\mathscr{BD}}(S)$ have the same category of germs from which $S$ is reconstructed via slices and ${\mathscr{BD}}(S)$ via bislices. In particular, a groupoidal \'etale range monoid $S$ can be reconstructed from the Boolean inverse monoid ${\mathscr{BD}}(S) = {\mathscr{I}}(S)$. This provides a different proof of a recent result by Lawson \cite{L24}.

\section*{Acknowledgements} I thank Richard Garner and Mark Lawson for useful conversations. I also thank the referee for a very thoughtful report and many useful comments and suggestions.

\end{document}